\documentclass[a4paper,11pt,reqno]{amsart}

\usepackage{enumitem}
\usepackage{amsmath,amsthm}
\usepackage{amssymb,latexsym}
\usepackage{amsaddr}
\usepackage{graphicx,subfig}
\usepackage{epsfig, mathtools}
\usepackage{color}
\usepackage{hyperref}

\newtheorem{theorem}{Theorem}[section]

\newtheorem{lemma}[theorem]{Lemma}

\newtheorem{proposition}[theorem]{Proposition}

\newtheorem{claim}[theorem]{Claim}

\newtheorem{conj}[theorem]{Conjecture}

\numberwithin{equation}{section}

\def\eps{\varepsilon}

\DeclareMathOperator{\dist}{dist}

\title{Long properly coloured cycles in edge-coloured graphs.}

\author{Allan Lo}
\thanks{The research leading to these results was partially supported by EPSRC, grant no. EP/P002420/1.}
\address{School of Mathematics, University of Birmingham,\\Birmingham, B15 2TT, UK}
\email{s.a.lo@bham.ac.uk}

\date{\today}
\begin{document}

\begin{abstract}
Let $G$ be an edge-coloured graph. 
The minimum colour degree $\delta^c(G)$ of $G$ is the largest integer $k$ such that, for every vertex~$v$, there are at least $k$ distinct colours on edges incident to~$v$.
We say that $G$ is properly coloured if no two adjacent edges have the same colour.
In this paper, we show that, for any $\varepsilon >0$ and $n$ large, every edge-coloured graph $G$ with $\delta^c(G) \ge  (1/2+\eps)n$ contains a properly coloured cycle of length at least $\min\{ n , \lfloor 2 \delta^c(G)/3 \rfloor\}$.
\end{abstract}

\maketitle

\section{Introduction}

An \emph{edge-coloured} graph is a graph~$G$ with an edge-colouring~$c$ of~$G$.
We say that $G$ is \emph{properly coloured} if no two adjacent edges of $G$ have the same colour. 
If all edges have the same (or distinct) colour, then $G$ is \emph{monochromatic} (or \emph{rainbow}, respectively).

Finding properly coloured subgraphs in edge-coloured graphs $G$ has a long and rich history. 
Grossman and H{\"a}ggkvist~\cite{MR701173} are the first to give a sufficient condition on the existence of properly coloured cycles in edge-coloured graphs with two colours.
Later on, Yeo~\cite{MR1438622} extended the result to edge-coloured graphs with any number of colours.
A natural question is to ask what guarantees the existence of properly coloured Hamiltonian paths and cycles.

In particular, the case when $G$ is an edge-coloured $K_n$ has been receiving the most attention. 
Given $k \in \mathbb{N}$, an edge-coloured graph $G$ is \emph{locally $k$-bounded} if for all vertices $v \in V(G)$, no colour appears more than $k$ times on the edges incident to~$v$ for all vertices~$v$.
A conjecture of Bollob\'as and Erd\H{o}s~\cite{MR0411999} states that every locally $(\lfloor n/2 \rfloor-1)$-bounded edge-coloured $K_n$ contains a properly coloured Hamilton cycle.
There is a series of partial results toward this conjecture by Bollob\'as and Erd\H{o}s~\cite{MR0411999}, Chen and Daykin~\cite{MR0422070}, Shearer~\cite{MR523092}, and Alon and Gutin~\cite{MR1610269}.
In~\cite{LoPCHCinKn} the author showed that the conjecture of Bollob\'as and Erd\H{o}s holds asymptotically, that is, for any $\eps >0$ and $n$ sufficiently large, every locally $(1/2 -\eps) n$-bounded edge-coloured $K_n$ contains a properly coloured Hamilton cycle.
A hypergraph generalisation of finding properly coloured Hamilton cycle in locally $k$-bounded edge-coloured complete graphs has also been studied by Dudek, Frieze and Ruci{\'n}ski~\cite{MR2900421} as well as Dudek and Ferrara~\cite{MR3338019}.
Recently, Sudakov and Volec~\cite{SudakovVolec} proved that every locally ${n}/(500 r^{3/4})$-bounded edge-coloured~$K_n$ contains all properly coloured graphs with at most $r$ paths of length two.
This proved a conjecture of Shearer~\cite{MR523092} as well as improves results of Alon, Jiang, Miller, Pritikin~\cite{MR2016871} and B{\"o}ttcher, Kohayakawa and Procacci~\cite{MR2925305}.
For a survey regarding properly coloured subgraphs in edge-coloured graphs, we recommend Chapter~16 of~\cite{MR2472389}.
Also see~\cite{KanoLiSurvey} for a survey regarding monochromatic and rainbow subgraphs in edge-coloured graphs.

Consider an edge-coloured (not necessarily complete) graph~$G$. 
Given a vertex $v \in V(G)$, the \emph{colour degree} $d^c_G(v)$ is the number of distinct colours of edges incident to $v$.
The \emph{minimum colour degree $\delta^c(G)$} is the minimum $d^c_G(v)$ over all vertices~$v$ in~$G$.
Li and Wang~\cite{MR2519172} showed that every edge-coloured graph~$G$ with $\delta^c(G) \ge d$ contains a properly coloured path of length $2 d$ or a properly coloured cycle of length at least $2 d/3$.
In~\cite{Lo10}, the author improved $2 d/3$ to $d +1$, which is best possible.
In the same paper, the author conjectured the following. 
\begin{conj} \label{conj}
Every edge-coloured connected graph $G$ with $\delta^c(G) \ge d$ contains a properly coloured Hamilton cycle or a properly coloured path of length $\lfloor 3 d/2 \rfloor$.
\end{conj}

If this conjecture holds, then the bound is sharp by the following example. 
Let $d, n \in \mathbb{N}$ with $n \ge 3d/2$.
Let $c_1, c_2, \dots, c_d$ be distinct colours. 
Let $X, Y$ be disjoint sets of vertices such that $X = \{x_1, x_2,\dots, x_d\}$ and $|Y| = n- d$. 
For each $1 \le i \le d$, join $x_i$ to each vertex of $Y$ with colour $c_i$.
For $1 \le i< j \le d$, join $x_i$ to $x_j$ with a new distinct colour. 
Let $G$ be the resulting edge-coloured graph. 
Note that $G$ has $n$ vertices and $\delta^c(G) = d$.
Every properly coloured path in~$G$ with both endpoints in $Y$ must contain at least two vertices in~$X$.
Thus, every properly coloured path in~$G$ is of length at most $|X| + \lfloor |X|/2 \rfloor = \lfloor 3d/2 \rfloor$.

In~\cite{LoDirac}, the author proved that the conjecture holds when $d \ge (2/3+\eps)n$ for $\eps >0$ and $n$ large, that is, every edge-coloured graph~$G$ on $n$ vertices with $\delta^c(G) \ge (2/3 +\eps) n$ contains a properly coloured Hamilton cycle. 

In this paper, we prove the following results.
\begin{theorem} \label{thm:pccycle}
For $\eps >0$, there exists $n_0 \in \mathbb{N}$ such that every edge-coloured graph~$G$ on $n \ge n_0$ vertices with $\delta^c(G) \ge (1/2 + \eps) n$ contains a properly coloured cycle of length at least $\min\{ \lfloor 3 \delta^c(G) /2 \rfloor, n\}$.
\end{theorem}
Note that Theorem~\ref{thm:pccycle} implies Conjecture~\ref{conj} when $d \ge (1/2+ \eps)n$ and $n$ large.
By analysing the proof of Theorem~\ref{thm:pccycle}, one might be able to prove Conjecture~\ref{conj} when $d \ge n/2$.
Therefore, it would be interesting to know whether Conjecture~\ref{conj} hold for $d <n/2$.

\section{Notation and sketch proof} \label{sec:notation}

For a graph~$G$, we denote $V(G)$ and $E(G)$ for the vertex set and edge set of $G$, respectively. 
Write $|G|$ for $|V(G)|$.
For (edge-coloured) graphs $G$ and $H$, we write $G-H$ for the graph with vertex set $V(G)$ and edge set $E(G) \setminus E(H)$.
For $W \subseteq V(G)$, we write $G \setminus W$ for the subgraph of $G$ induced by the vertex set $V(G) \setminus W$, and write $G \setminus H$ for $G \setminus V(H)$.
For disjoint $X,Y \subseteq V(G)$, let $G[X]$ be the (edge-coloured) subgraph induced by~$X$ and let $G[X,Y]$ be the induced bipartite subgraph with vertex classes $X$ and $Y$.
For a set of edges~$E$, we write $G \cup E$ for the graph with vertex set $V(G) \cup V(E)$ and edge set $E(G) \cup E$.
For a singleton set $\{v\}$, we sometimes write $v$ for short. 

For an edge-coloured graph~$G$, let $C(G) : = \{ c(uv) : uv \in E(G)\}$, that is, the set of colours appeared in~$G$. 
For a vertex $v \in V(G)$, let $C_G(v) : = \{ c(uv) : u \in N_G(v)\}$.
Thus $d^c_G(v) = |C_G(v)|$.
For $V \subseteq V(G)$, define $d^c_G(v,V) : = | C_{G[V \cup v]} (v)|$.
Let ${\bf x} = (x,c_x)$ be a pair with vertex $x \in V(G)$ and colour $c_x \in C_G(x)$.
We write $N_G({\bf x})$ be the set of vertices $v \in N_G(x)$ such that $c(xv) \ne c_x$.
For distinct $x,y \in V(G)$, we denote by \emph{${\rm dist}_G(x,y)$} the shortest distance between $x$ and~$y$. 
If $x$ and $y$ are not connected, then we say ${\rm dist}_G(x,y) = \infty$.
If $G$ is known from the context, then we omit $G$ in the subscript.

For a path $P = x_1 x_2 \dots x_k$ from $x_1$ to $x_{k}$ and a vertex $y \notin V(P)$, we write $Py$ for the path $x_1 x_2 \dots x_k y$.
If $P' = y_1 \dots y_{\ell}$ is a path with $y_1 = x_k$ and $V(P) \cap V(P') = \{x_k\}$, then we write $PP'$ for the concatenated path $x_1 x_2 \dots x_k y_2 \dots y_{\ell}$.

An edge-coloured graph~$G$ is \emph{critical}, if for every edge~$uv$, $d^c_{G}(u) > d^c_{G-uv}(u)$ or $d^c_{G}(v) > d^c_{G-uv}(v)$.
Note that if $G$ is critical, then any monochromatic subgraph~$H$ of~$G$ is a union of vertex-disjoint stars.
Since we are only concerning about properly coloured subgraphs, we may assume further that any two vertex-disjoint monochromatic component in~$G$ have distinct colours. 
Thus, from now on, we assume that every monochromatic subgraph~$H$ of any critical edge-coloured graph~$G$ is a star.

Let $F$ be a direct graph. 
For $u,v \in V(F)$, we write $uv$ for the directed edge from $u$ to~$v$.
For $Z,W \subseteq V(F)$, denote by $e_F(Z, W)$ the number of directed edges from $Z$ to $W$ in~$F$.

The constants in the hierarchies used to state our results are chosen from right to left.
For example, if we claim that a result holds whenever $0<1/n\ll a\ll b\ll c\le 1$ (where $n$ is the order of the graph), then there is a non-decreasing function $f:(0,1]\to (0,1]$ such that the result holds
for all $0<a,b,c\le 1$ and all $n\in \mathbb{N}$ with $b\le f(c)$, $a\le f(b)$ and $1/n\le f(a)$. 
Hierarchies with more constants are defined in a similar way.

\subsection{Sketch proof of Theorem~\ref{thm:pccycle}}

Here we present an outline of the proof of Theorem~\ref{thm:pccycle}, which naturally splits into three lemmas. 
First, we consider the case when $G$ is close to the extremal example in Section~\ref{sec:extremalcase}.
More precisely, for $\delta, \eps >0$, we say that an edge-coloured graph $G$ on $n$ vertices is \emph{$(\delta, \eps)$-extremal} if there exist disjoint $A,B \subseteq V(G)$ such that
\begin{enumerate}[label={\rm(A\arabic*)}]
	\item 
		\label{itm:A1}
	$|A| \ge (\delta - \eps) n $ and $|B| \ge (1-\delta - \eps )n$;
	\item 
		\label{itm:A2}
	for each $ a \in A$, there exists a distinct colour $c_a$ such that there are at least $|B| - \eps n$ vertices $b \in B$ such that $c(ab) = c_a$;
	\item 
		\label{itm:A3}
	for each $b \in B$, $d_G(b) \le (\delta + \eps ) n $ and $b$ has at least $|A|- \eps n$ neighbours $a \in A$ such that $c(ab) = c_a$. 
\end{enumerate}
Throughout this paper, we will always assume that $\eps \ll \delta$.
In this case, we will find a properly coloured cycle (of the desired length) directly (see Section~\ref{sec:extremalcase}). 
\begin{lemma} \label{lma:extremal}
Let $0 < 1/n \ll \eps \ll \delta \le 1$.
Let $G$ be a $(\delta, \eps)$-extremal critical edge-coloured graph on $n$ vertices with $\delta^c(G) \ge \delta n$.
Then $G$ contains a properly coloured cycle of length $\min\{ \lfloor 3 \delta n /2 \rfloor , n \}$.
\end{lemma}

Note that Lemma~\ref{lma:extremal} does not require that $\delta \ge 1/2+\eps$. 
Thus Lemma~\ref{lma:extremal} implies that Conjecture~\ref{conj} holds if $G$ is $(\delta,\eps)$-extremal with $1/n \ll \eps \ll \delta\le 1$.

If $G$ is not close to the extremal, then we proceed using the \emph{absorption technique} introduced by R\"odl, Ruci\'{n}ski and Szemer\'{e}di~\cite{MR2399020}, which was used to tackle Hamiltonicity problems in hypergraphs.
The absorption technique has been adapted for finding properly coloured Hamilton cycles in~\cite{LoDirac,LoPCHCinKn}. 
First we find  a small `absorbing cycle'~$C$ in~$G$ using the following lemma, which is proved in Section~\ref{sec:abs}.

\begin{lemma} \label{lma:abscycle}
Let $0 < 1/n \ll \gamma \ll \eps < 1/2$.
Suppose that $G$ is an edge-coloured graph on $n$ vertices with $\delta^c(G) \ge (1/2 + \eps) n$.
Then there exists a properly coloured cycle~$C$ of length at most $ \eps n /2$ such that for any collection $P_1,\dots,P_k$ of vertex-disjoint properly coloured paths in $G \setminus V(C)$ with $k \le \gamma n$, there exists a properly coloured cycle with vertex set $V(C) \cup \bigcup_{1 \le i \le k} V(P_i)$.
\end{lemma}

Remove the vertices of $C$ from $G$ and call the resulting graph~$G'$. 
Since $G$ is not extremal, neither is~$G'$.
(Indeed, if $G'$ is $(\delta, \eps)$-extremal with vertex subsets $A,B$, then $G$ is $(\delta, 2 \eps)$-extremal with vertex subsets $A,B$ as $\eps \ll 1$.)
We find vertex-disjoint properly coloured paths by the next lemma (which is implied by Lemma~\ref{lma:1-path-cycle}).

\begin{lemma} \label{lma:pathcovers}
Let $0 < 1/n \ll \beta \ll \eps \ll 1/2 < \delta$.
Suppose that $G$ is a critical edge-coloured graph on $n$ vertices with $\delta^c(G)  \ge \delta n+1 $.
If $G$ is not $(\delta, \eps)$-extremal, then $G$ contains vertex-disjoint properly coloured paths $P_1, \dots, P_{k}$  with $k \le 100 \beta^{-1}$ covering $\min \{ ( 3 \delta + \beta )n/2, n \}$ vertices. 
\end{lemma}

We now prove Theorem~\ref{thm:pccycle} using Lemmas~\ref{lma:extremal}--\ref{lma:pathcovers}.

\begin{proof}[Proof of Theorem~\ref{thm:pccycle}]
Without loss of generality, we may assume that $G$ is critical edge-coloured with $\delta^c(G) = \delta n$ and that $\eps$ is sufficiently small. 
Let $\gamma, \eps'$ be such that $1/n \ll \gamma \ll \eps \ll \eps' \ll 1/2 <\delta$.

Apply Lemma~\ref{lma:abscycle} and obtain a properly coloured cycle $C$ of length at most $\eps n /2$ such that for any collection $P_1,\dots,P_k$ of vertex-disjoint properly coloured paths in $G \setminus V(C)$ with $k \le \gamma n$, there exists a properly coloured cycle with vertex set $V(C) \cup \bigcup_{1 \le i \le k} V(P_i)$.

Let $G':= G \setminus C$, $n' := |G'|$ and $\delta' :=  (\delta n - |C|-1)/n'$.
Note that $\delta^c (G) \ge \delta' n' +1$ and $1/n' \ll \eps \ll \eps' \ll 1/2 < \delta'$.
If $G'$ is not $(\delta', \eps')$-extremal, then apply Lemma~\ref{lma:pathcovers} (with $\eps, \eps',\delta', n'$ playing the roles of $\beta,\eps,\delta,n$) and obtain vertex-disjoint properly coloured paths $P_1, \dots, P_{k}$ such that  $k \le 100 \eps^{-1} \le \gamma n $ and 
\begin{align*}
\bigcup_{i \le k } |V (P'_i) |  \ge \min \{ 3(\delta - |C|-1)n + \eps n' ) /2 , n - |C| \} \ge  \min \{ 3 \delta n/2 , n \}- |C|
\end{align*}
as $|C| \le \eps n/2 \le \eps n'$.
Thus, by the property of~$C$, there exists a properly coloured cycle $C'$ with vertex set $V(C) \cup \bigcup_{ i \le k} V(P'_i)$.
So $|C'| \ge \min \{ 3 \delta n/2 , n \}$ as desired.

On the other hand, if $G'$ is $(\delta', \eps')$-extremal, then there exist disjoint $A,B \subseteq V(G') = V(G) \setminus V(C) $ satisfying 
\begin{enumerate}[label={\rm(A\arabic*)}]
	\item 

	$|A| \ge (\delta' - \eps') n' \ge ( \delta - 2 \eps' ) n$ and $|B| \ge (1-\delta' - \eps' )n' \ge (1- \delta - 2 \eps' )n$;
	\item
	for each $ a \in A$, there exists a colour $c_a$ such that there are at least $|B| - \eps' n' \ge |B| - 2 \eps' n$ vertices $b \in B$ such that $c(ab) = c_a$;
	\item
	for each $b \in B$, 
	\begin{align*}
	d_G(b) \le d_{G'}(b) + |C| 
	\le (\delta' + \eps' ) n' + |C| 
	= \delta n -1  + \eps' n' 
	< (\delta + 2 \eps')n
	\end{align*}
	and $b$ has at least $|A|- \eps' n' \ge |A| - 2\eps' n $ neighbours $a \in A$ such that $c(ab) = c_a$. 
\end{enumerate}
Therefore $G$ is $(\delta, 2\eps')$-extremal.
By Lemma~\ref{lma:extremal}, $G$ contains a properly coloured cycles of length at least $\min \{ \lfloor 3 \delta n/2 \rfloor, n \}$.
\end{proof}

\section{Extremal case} \label{sec:extremalcase}

In this section, we prove Lemma~\ref{lma:extremal}, that is, Theorem~\ref{thm:pccycle} when $G$ is critical and $(\delta, \eps)$-extremal.
We would need the following definition. 
Let $G$ be an edge-coloured graph on $n$ vertices.
Let $A,B \subseteq V(G)$ be disjoint.
We say that the ordered pair $(A,B)$ is \emph{$\eps$-extremal} if the following holds:
\begin{enumerate}[label={\rm(E\arabic*)}]
	\item for each $ a \in A$, there exists a distinct colour $c_a$; \label{itm:E1}
	\item for each $a \in A$, there are at least $|B| - \eps n$ vertices $b \in B \cap N(a)$ such that $c(ab) = c_a$, and at least $|A| - \eps n $ vertices $a' \in A \cap N(a)$ such that $c_a \ne c(aa') \ne c_{a'}$; \label{itm:E2}
	\item for each $ b \in B$, there are at least $|A| - \eps n$ vertices $a \in A \cap N(b)$ such that $c(ab) = c_a$. \label{itm:E3}
\end{enumerate}
Next we show that if $G$ is $(\delta, \eps)$-extremal, then there exists $4 \sqrt{ \eps }$-extremal pair in~$G$.

\begin{lemma} \label{lma:extremalH}
Let $0 < 1/n \ll \eps \ll 1$ and let $\delta > 4 \sqrt{\eps}$.
Let $G$ be a critical edge-coloured graph on $n$ vertices with $\delta^c(G) \ge \delta n$.
Suppose that $G$ is $(\delta, \eps)$-extremal.
Then there exist disjoint $A,B \subseteq V(G)$ such that $(A,B)$ is $4 \sqrt{\eps}$-extremal, $|A| \ge (\delta - 4 \sqrt{\eps}) n $, $ |B|  \ge (1-\delta - \eps )n$ and, for each $b \in B$, $d_{G} ( b ) \le ( \delta + \eps ) n$.
\end{lemma}

\begin{proof}
Let $\eps' := 4 \sqrt{\eps} $.
Since $G$ is $(\delta, \eps)$-extremal, there exist disjoint $A^* , B^* \subseteq V(G)$ satisfying~\ref{itm:A1}--\ref{itm:A3}.

Note that $|V(G) \setminus (A^* \cup B^*)| \le 2 \eps n$.
We say that an edge $aa'$ in $G[A^*]$ is \emph{good} if $c_a \ne c(a a') \ne c_{a'}$.
We bound the number of good edges from below as follows.
Define a directed graph~$D$ on~$A^*$ such that there is a directed edge from $a$ to $a'$ if and only if $c_a \ne c(a a')$.
For each $a \in A^*$, $a$ sends at most $1+\eps n + |V(G) \setminus (A^* \cup B^*)| \le 3 \eps n+1$ distinct colours (including the colour $c_a$) to $V(G) \setminus A^*$ by~\ref{itm:A2}.
So the outdegree of~$a$ in~$D$ is at least $\delta n - 3 \eps n-1 \ge |A^*| - 5 \eps n -1$.
Since the number of good edges equals the number of $2$-cycles in~$D$, the number of good edges is at least $( |A^*| - 5 \eps n  -1) |A^*| - \binom{|A^*|}2= |A^*| ( |A^*| - 10 \eps n-1)/2$.
Let $A'$ be the set of $a \in A^*$ that is incident with at most $ |A^*| - \eps' n $ good edges.
Note that $|A'| \le 3 \sqrt{\eps} n $.

Let $A := A^* \setminus A'$.
Thus $|A| \ge |A^*| - 3 \sqrt{\eps} n \ge (\delta - \eps') n$ by~\ref{itm:A1}.
Moreover, every $a \in A$ is incident with at least $|A| - \eps' n $ good edges in~$G[A]$ implying~\ref{itm:E2}.
Set $B := B^*$. 
So $|B| \ge (1- \delta- \eps) n $.
Also, \ref{itm:A3} implies that~\ref{itm:E3} holds and that, for each $b \in B$, $d_{G} ( b ) \le ( \delta + \eps ) n$.
Therefore $(A,B)$ is $\eps'$-extremal.
\end{proof}

In the next two lemma, we find properly coloured cycles spanning $A \cup B$, when $(A, B)$ is $\eps$-extremal.

\begin{lemma} \label{lma:extremalHcycle}
Let $\eps < 1/36$.
Let $G$ be an edge-coloured graph on $3m$ vertices.
Suppose that there is a partition $A,B$ of $V(G)$ such that $(A,B)$ is $\eps$-extremal, $|A| = 2m$ and $|B| = m $.
Then $G$ has a properly coloured Hamilton cycle. 
\end{lemma}

\begin{proof}
Partition $A$ into $X$ and $Y$ each of size~$m$.
Let $H_X$ be the subgraph of $G[X,B]$ induced by edges of colour in~$\{ c_a : a \in A \}$.
By \ref{itm:E2} and~\ref{itm:E3}, $H_X$ is a bipartite graph with $\delta(H_X) \ge m - 3 \eps m$.
Hence by Hall's theorem, there exists a perfect matching $M_X$ in~$H_X$.

Similarly, let $H_Y$ be the subgraph of $G[Y,B]$ induced by edges of colour in~$\{ c_a : a \in A \}$ and there exists a perfect matching $M_Y$ in~$H_Y$.
Note that $M_X \cup M_Y$ is a union of $m$ vertex-disjoint path of length~$2$ each with midpoint in~$B$.
By \ref{itm:E1}, $M_X \cup M_Y$ is properly coloured.
Let $M_X \cup M_Y = \{ x_ib_iy_i : x_i \in X, b_i \in B, y_i \in Y \text{ and } i \le m\}$.

Now define an oriented graph $F$ on vertex set $Z = \{ z_1, \dots, z_m\}$ such that there is a directed edge from $z_i$ to $z_j$ if and only if $y_i x_j$ is an edge (in $G$) with $c_{y_i} \ne c(y_i x_j) \ne c_{x_j}$.
By~\ref{itm:E2}, each $z_i$ has indegree and outdegree at least $m-  3 \eps m \ge m/2$.
Therefore $F$ contains a directed Hamilton cycle by a result of Ghouila-Houri~\cite{GH}, $z_1 z_2 \dots z_m z_1$ say.
Then $x_1 b_1 y_1 x_2 b_2 y_2 \dots z_m x_1$ is a properly coloured Hamilton cycle in~$G$ as desired.
\end{proof}

\begin{lemma} \label{lma:pathsystem}
Let $\ell \in \mathbb{N}$ and $0 < 1/n \ll  \eps \ll \alpha<  1/3$ with $\ell < \alpha  n$.
Let $G$ be a critical edge-coloured graph on $n$ vertices.
Suppose that $(A,B)$ is $\eps$-extremal such that $\alpha n + \ell +1 \le |B| \le |A|/2 +\ell$.
Suppose that $\mathcal{P}$ is a union of $\ell$ vertex-disjoint properly coloured paths such that each path has both of its endpoints in~$B$ and $| ( A \cup B ) \cap V ( \mathcal{P} ) | = 2 \ell$.
Then $G$ contains a properly coloured cycle with vertex set $V(C) = A \cup B \cup V(\mathcal{P})$.
\end{lemma}

\begin{proof}
First suppose that $|B| < |A|/2 +\ell $.
Let $p := |A|  -  2 ( |B| - \ell-1 )$, so $3 \le p \le |A| - 2 \alpha n $.
By~\ref{itm:E2} and a greedy argument, $G$ contains a properly colour path $b a_1 a_2 \dots a_{p} b'$ such that $a_1,\dots, a_p \in A$ and $b,b' \in B \setminus V ( \mathcal{P} )$.  
We add the path $b a_1 a_2 \dots a_{p} b'$ to $\mathcal{P}$ and call the resulting set $\mathcal{P}'$.
Let $A' = A \setminus \{a_1, \dots, a_p \}$, so $|A'| = |A|  - p = 2  ( |B|  - \ell  - 1)$.
Furthermore $(A', B)$ is $\eps$-extremal.
Therefore by replacing $A,B,\mathcal{P}$ with $A',B,\mathcal{P}'$, we may assume that without loss of generality that $|A| = 2m$ and $|B| = m + \ell$ for some integer $m \ge \alpha n$ with $\ell \le m$.

Consider $G[A \cup B] \cup \mathcal{P}$.
Suppose that $P_1, \dots, P_{\ell}$ are the paths of~$\mathcal{P}$.
We now contract each $P_i$ as follows. 
Let $b_i$ and $b_i'$ be the end vertices of~$P_i$, so $b_i,b_i' \in B$. 
Let $N_i$ be the common neighbours $a \in A$ of $b_i$ and $b_i'$ such that $c(a b_i ) = c(a b_i' ) =  c_a \notin C_{P_i}(b_i) \cup C_{P_i}(b_i')$.
Note that $|N_i| \ge |A| - 2 \eps n - 2 \ge 2m - 3 \eps \alpha^{-1} m \ge 2m - 3 \sqrt{\eps} m $ by~\ref{itm:E3}. 
We replace each $V(P_i)$ with a new vertex $x_i$ and join $x_i$ to each vertices $a \in N_i$ with colour $c_a$.
Call the resulting graph~$H$.
So $A \subseteq H$ and $|H| = 3m$.
Note that, for each $i \le \ell$, $d_H(x_i,A)= |N_i|\ge 2m - 3\sqrt{\eps} m$.
Since $V(H) \setminus A = B \setminus V( \mathcal{P} ) \cup \{ x_1, \dots, x_{\ell} \}$, it is easy to see that $(A,V(H) \setminus A)$ is $\sqrt{\eps}$-extremal in~$H$.
Lemma~\ref{lma:extremalHcycle} implies that $H$ has a properly coloured Hamiltonian cycle~$C$.
By replacing each $x_i$ in~$C$ with~$P_i$ we obtain a properly coloured cycle in $G$ with vertex set $A \cup B \cup V( \mathcal{P})$ as required.
\end{proof}

By Lemmas~\ref{lma:extremalH} and~\ref{lma:pathsystem}, to prove Lemma~\ref{lma:extremal} it suffices to find a union of suitable properly coloured paths.
We would need a finer partition $V(G) \setminus (A \cup B)$ into $Y$ and $Z$ as follows.
Let $Y$ be the set of $v \in V(G) \setminus ( A \cup B)$ such that $d^c_G(v,B) \ge 10 \eps n $ or $| \{  c(av) : a \in N_G(v) \cap A $ and $c(av) \ne c_a \} | \ge 10 \eps n $.
Let $Z := V(G) \setminus (A \cup B \cup Y)$.

\begin{proposition} \label{prop:BZ}
Let $\eps,\delta >0$.
Let $G$ be a critical edge-coloured graph on $n$ vertices with $\delta^c(G) \ge \delta n$.
Suppose that $(A,B)$ is $\eps$-extremal such that $|A| \ge (\delta - \eps) n $ and $ |B|  \ge (1-\delta - \eps )n$.
Let $Y,Z$ be a partition of $V(G) \setminus (A \cup B)$ as above.
For each $v \in Z$, there are at least $|A| - 24 \eps n$ vertices $a \in N_G(v) \cap  A$ such that $c(av) = c_a$.
Moreover, $(A,B \cup Z)$ is $24 \eps$-extremal.
\end{proposition}

\begin{proof}
Note that $|Y| + |Z| \le 2 \eps n $.
Consider any $v \in Z$.
Since $d^c_G(v,B) < 10 \eps n $, we have 
\begin{align*}
d^c_G(v,A) \ge d^c_G(v) - d^c_G(v,B ) - |Y| - |Z| \ge (\delta - 12 \eps) n \ge |A| - 14 \eps n .
\end{align*}
On the other hand, $| \{  c(av) : a \in N_G(v) \cap A $ and $c(av) \ne c_a \} | < 10 \eps n $.
Thus there are at least $|A| - 24 \eps n$ vertices $a \in N_G(v) \cap  A$ such that $c(av) = c_a$.
\end{proof}

Instead of finding a union of suitable properly coloured paths, the next lemma shows that finding a suitable matching is sufficient.

\begin{lemma} \label{lma:matching1}
Let $0 < 1/n \ll  \eps \ll \alpha <  1/3$.
Let $G$ be a critical edge-coloured graph on $n$ vertices.
Suppose that $(A,B)$ is $\eps$-extremal such that $|A| \ge (2 \alpha  + 6 \eps) n+2  $ and $ |B|  \ge (\alpha +4 \eps) n + 1$.
Let $Y$ be the set of $v \in V(G) \setminus ( A \cup B)$ such that $d^c_G(v,B) \ge 10 \eps n $ or $| \{  c(av) : a \in N_G(v) \cap A $ and $c(av) \ne c_a \} | \ge 10 \eps n $.
Let $Z:= V(H) \setminus (A \cup B \cup Y)$.
Suppose that $M$ and $M'$ are vertex-disjoint matchings such that 
\begin{enumerate}[label = {\rm (\roman*)}]
	\item there are at most $2 \eps n $ edges in $M \cup M'$;
	\item $M \subseteq G \setminus A $;
	\item $M'  \subseteq G[A , B \cup Z]$ and for each edges $av \in M'$ with $a \in A$, $c(av) \ne c_a$.
\end{enumerate}
Then $G$ contains a properly coloured cycle $C$ such that
\begin{align*}
|C|\ge \min \left\{ n, \left\lfloor \frac{3|A|}2 + |M| +\frac{|M'|}2 + |Y| - \frac{| V(M) \cap Y|}2 \right\rfloor \right\}.
\end{align*}
\end{lemma}

\begin{proof}
Note that $(A, B \cup Z)$ is $24 \eps $-extremal by Proposition~\ref{prop:BZ}.
Our aim is to extend $M \cup M'$ into a suitable path system~$\mathcal{P}$ (see Claim~\ref{clm2} for the precise properties) such that we can apply Lemma~\ref{lma:pathsystem}.
The key features of~$\mathcal{P}$ are that every path is properly coloured with both endpoints in $B \cup Z$ and that $\mathcal{P}$ covers~$Y$.
Here, we give a rough outline on how to construct $\mathcal{P}$ from $M \cup M'$ (that is, the proof of Claim~\ref{clm2}).
For simplicity, we assume that $M \subseteq G[B \cup Z]$ (so the edges of~$M$ can be already viewed as paths with both endpoints in $B \cup Z$).
For each edge $av \in M'$ with $a \in A$, we add the edge $ab$ with $b \in B$ such that $c(ab) = c_a \ne c(av)$.
In order to cover $Y$, consider any $y \in Y$. 
If $d^c_G(y,B) \ge 10 \eps n $, then we extend $y$ to a path $byb'$ with $b,b' \in B$.
Otherwise, we have $| \{  c(av) : a \in N_G(v) \cap A $ and $c(av) \ne c_a \} | \ge 10 \eps n $, so we construct the path $baya'b'$ with $a,a' \in A$ and $b,b' \in B$.

We now give the formal definition of $\mathcal{P}$ in the following claim.

\begin{claim} \label{clm2}
Let $q:= | V(M) \cap Y|$.
There exists a properly coloured subgraph~$\mathcal{P}$ of~$G$ such that 
$M \cup M' \subseteq \mathcal{P}$ and 
\begin{enumerate}[label = {\rm (\roman*$'$)}]
	\item $\mathcal{P}$ is a union of $\ell^*$ vertex-disjoint path such that each path has both endpoints in $B \cup Z$;
	\item $\ell^* = |M| + |M'| + |Y| -q \le 4 \eps n $;
	\item $\mathcal{P}$ covers $Y$;
	\item $\mathcal{P}$ contains precisely $2 \ell^*$ vertices in $B \cup Z$, that is, each vertex in $V ( \mathcal{P} ) \cap (B \cup Z)$ is an endpoint of some path in~$\mathcal{P}$;
	\item $\mathcal{P}$ contains at most $|M'|  + 2|Y| - q$ vertices in $A$.
\end{enumerate}
\end{claim}

\begin{proof}[Proof of claim]
We construct $\mathcal{P}_0$ as follows.
Initially, we set $\mathcal{P}_0 : = M \cup M'$.
For each edge $av \in M'$ with $a \in A$, we add an edge $ab$ to $\mathcal{P}_0$ such that $b \in B \setminus V(\mathcal{P})$ is distinct and $c(ab) = c_a \ne c(av)$ (which exists by~\ref{itm:E2}).
Thus $\mathcal{P}_0$ is a union of $|M| +|M'|$ vertex-disjoint paths such that each path has both endpoints in $V(G) \setminus A$,
\begin{align*}
	| V( \mathcal{P}_0 )  \setminus A | = 2 |M| + 2 |M'|, \quad 
	| V(\mathcal{P}_0) \cap Y| = q
	\quad 	\text{ and } \quad
	|V( \mathcal{P} ) \cap A| = |M'|.
\end{align*}

Let $Y := \{y_1, \dots, y_{|Y|} \}$ be such that $V(\mathcal{P}_0) \cap Y = \{y_1, \dots, y_{q}\}$.
Suppose that for some $ i \le |Y|$ we have already constructed $\mathcal{P}_0 \subseteq \dots \subseteq \mathcal{P}_{i-1}$ such that for all $j < i$ 
\begin{enumerate}[label={\rm(Q\arabic*)}]
	\item 
	\label{itm:Q1}
	$\mathcal{P}_j$ is an union of $|M| + |M'| + \max \{ 0, j - q\}$ vertex-disjoint properly coloured paths;
	\item 
	\label{itm:Q2}
	$| (B \cup Z) \cap V(\mathcal{P}_{j}) | = 2|M| + 2|M'|  -q +  j + \max \{ 0, j - r\}$ and $|A \cap V(\mathcal{P}_{j})| \le |M'| + j + \max \{ 0, j - q\}$;
	\item 
	\label{itm:Q3}
	every vertex in $V(\mathcal{P}_{j}) \cap (B \cup Z)$ is an endpoint of some paths in $\mathcal{P}_j$;
	\item 
	\label{itm:Q4}
	for all $j' \le j$, $d_{\mathcal{P}_{j}} (y_{j'}) = 2$ and for all $j' > j$, $d_{\mathcal{P}_{j}} (y_{j'}) = d_{\mathcal{P}_{j-1}} (y_{j'}) $.
\end{enumerate}
We now construct $\mathcal{P}_i$ as follows.
By~(Q2), $|B \cap  V(\mathcal{P}_{i-1}) | , |A \cap  V(\mathcal{P}_{i-1})| \le 8 \eps n $.

Note that by~(Q4)
\begin{align*}
d_{\mathcal{P}_{i-1}} (y_{i}) = d_{ \mathcal{P}_{0}} (y_{i}) = d_M(y_{i}) = 
\begin{cases}
1 & \text{if $i \le q$}\\
0 & \text{otherwise.}
\end{cases}
\end{align*}

Suppose that $i \le q$.
Let $c'$ be the colour of the edge incident with $y_{i}$ in~$\mathcal{P}_{i-1}$.
If $d_G^c(y_i, B ) \ge 10 \eps n$, then there exists an edge $b y_{i}$ such that $b \in B \setminus V(\mathcal{P}_{i-1})$ and $c(b y_{i}) \ne c'$ and set $\mathcal{P}_{i} : = \mathcal{P}_{i-1} \cup b y_{i}$.
Thus, we may assume that there exist at least $10 \eps n$ vertices $a \in A \cap N_G(y_{i})$ such that $ c(a y_{i}) \ne c_a$ and these $c(a y_{i})$ are distinct.
So there exists a vertex $a \in ( A \cap N_G(y_{i}) ) \setminus V(\mathcal{P}_{i-1})$ such that $c_a \ne c(a y_{i}) \ne c'$.
By~\ref{itm:E2}, there exists a vertex $b \in B \cap N_G(a) \setminus V(\mathcal{P}_{i-1})$ such that $c(ab) = c_a \ne c(ay_{i})$.
Set $\mathcal{P}_{i} := \mathcal{P}_{i-1} \cup \{a y_{i} , ab \}$.
A similar argument also holds for the case when $i > q$, where we apply the previous argument twice. 
Finally, set $\mathcal{P} := \mathcal{P}_{|Y|}$.
\end{proof}

Let $A^* : = A \setminus V(\mathcal{P})$.
Let $B^*$ be a subset of $B \cup Z$ such that $V(\mathcal{P}) \cap ( B \cup Z) \subseteq B^*$ and $|B^*| = \min \{ |B |+| Z|, \lfloor |A^*| /2 \rfloor + \ell^* \}$.

Note that $|B | \ge (\alpha +4 \eps n)+1 \ge \alpha n + \ell^* + 1$, where the last inequality holds by Claim~\ref{clm2}(ii$'$).
Since $|Y| \le 2 \eps n$,  together with Claim~\ref{clm2}(v$'$) and (i), we have
\begin{align*}
	|A^*|  \ge |A| - ( |M'| + 2 |Y| ) \ge  |A| - 6 \eps n \ge 2 \alpha n +2.
\end{align*}	
Therefore, we deduce that $|B^*| \ge \alpha n + \ell^* + 1$.

Note that $(A^*,B^*)$ is $24\eps$-extremal (as $(A, B \cup Z)$ is by Proposition~\ref{prop:BZ}).
By Lemma~\ref{lma:pathsystem}, $G$ contains a properly coloured cycle~$C$ with vertex set $A^* \cup B^* \cup V( \mathcal{P} ) = A \cup B^* \cup Y$ by Claim~\ref{clm2}(iii$'$).
If $|B^*| = |B|+ |Z|$, then $C$ is a properly coloured Hamilton cycle of~$G$. 
If $|B^*| = \lfloor |A^*| /2 \rfloor + \ell^* $, then
\begin{align*}
	|C| & = 
	|A| +|Y|+ |B^*|
	= |A|  + |Y|+ \lfloor |A^*| /2 \rfloor + \ell^*  \\
	&  = |A|  + |Y|+ \lfloor (|A| - |V(\mathcal{P}) \cap A|) /2 \rfloor + \ell^*  \\
	& \overset{\mathclap{\text{(ii$'$), (v$'$)}}}{\ge} |A|  + \left\lfloor \frac{ |A| - ( |M'| + 2|Y| - q ) }{2} \right\rfloor + |M| + |M'| + 2|Y|-q \\
	& = \left\lfloor \frac{3|A|}2 +|M| + \frac{|M'|}2 + |Y| -\frac{q}2 \right\rfloor
\end{align*}	
as required.
\end{proof}

We are ready to prove Lemma~\ref{lma:extremal}.

\begin{proof}[Proof of Lemma~\ref{lma:extremal}]
Let $\eps' := 4 \sqrt{\eps}$ and without loss of generality (by adjusting $\eps'$ slightly), we have $(\delta - \eps') n \in \mathbb{Z}$.
Let $\alpha$ such that $\eps \ll \alpha \ll \delta$.
Apply Lemma~\ref{lma:extremalH} and obtain an $\eps'$-extremal pair $(A,B)$ such that  $|A| \ge (\delta - \eps') n $,
\begin{align*}
|B|  \ge (1-\delta - \eps' )n \ge (\alpha + 8\eps')n+1.
\end{align*}
and 
\begin{align}
\label{eqn:e1}
d_{G} ( b ) & \le ( \delta + \eps ) n \text{ for each $b \in B$}.
\end{align}
By removing vertices of~$A$ if necessary, we may assume that 
\begin{align}
|A| = (\delta - \eps') n  \ge (2\alpha + 12\eps')n+2. \label{eqn:e2}
\end{align}
Let $Y$ be the set of $v \in V(G) \setminus ( A \cup B)$ such that $d^c_G(v,B) \ge 10 \eps' n $ or $| \{  c(av) : a \in N_G(v) \cap A $ and $c(av) \ne c_a \} | \ge 10 \eps' n $.
Let $Z := V(G) \setminus (A \cup B \cup Y)$.
Let $p := \max \{ \eps' n - |Y| , 0\}$, so 
\begin{align}
	|Y| \ge \eps'n -p. \label{eqn:Y}
\end{align}

Let $F: = G \setminus A$.
So $\delta ( F ) \ge \eps' n$.
Let $R$ be the set of vertices $v \in V (F)$ such that $d_F (v) \le 10 \eps' n$ and let $S := V(F) \setminus R$.
Note that $|R| \ge (1- \delta - \eps')n$ as $B \subseteq R$ by~\ref{itm:E3} and~\eqref{eqn:e1}. 
Since $\Delta (F[R]) \le 10 \eps' n $, Vizing's theorem implies that there exists a matching $M_R$ in $F[R]$ such that $|M_R| \ge e(F[R]) / ( 10 \eps' n + 1 ) \ge 8 e(F[R]) / |R|$.
By summing the degrees $d_F(v)$ in $v \in R$, we have
\begin{align}
	|R| \eps' n & \le \sum_{v \in R} d_{F}(v) = 2 e(F[R]) + e(F[R,S])
	\le |R| |M_R| / 4 + |R||S|, \nonumber \\
	\eps' n & \le  |M_R| /4 + |S|. \label{eqn:M_R}
\end{align}
We now divide the proof into two different cases.

\medskip \noindent
\textbf{Case 1: $|M_R| + |S| \ge   \eps' n + p /2 $}.
We claim that there exists a matching $M$ in $F= G\setminus A$ such that $|M| = \lceil \eps' n + p /2 \rceil $.
Indeed, there is nothing to prove if $|M_R| \ge \eps' n + p/2$.
If $|M_R| < \eps' n + p/2$, then we can extend $M_R$ into a matching $M$ of size $\lceil \eps' n + p /2 \rceil$ by adding (appropriate) edges incident with $S$ (as $d_F(s) \ge 10 \eps' n $ for all $s \in S$ and $p \le \eps' n$). 

Note that $|M| = \lceil \eps' n + p /2 \rceil  \le 2 \eps' n $ and  
\begin{align*}
	& \left\lfloor \frac{3|A|}2 + |M| + |Y| -\frac{|V(M_R) \cap Y|}2 \right\rfloor
	 \ge \left\lfloor \frac{3|A|}2 + |M| + \frac{|Y|}2 \right\rfloor \\
	& \overset{ {\eqref{eqn:e2}, \eqref{eqn:Y}} }{\ge} \left\lfloor \frac{3( \delta - \eps' ) n }2 + \eps' n + \frac{p}2 + \frac{\eps' n - p}{2}  \right\rfloor = \lfloor 3 \delta n /2 \rfloor.
\end{align*}
By Lemma~\ref{lma:matching1} (with $M,\emptyset,\eps'$ playing the roles of $M,M', \eps$), $G$ contains a properly coloured cycle~$C$ such that  $|C| \ge \min \{ n,  \lfloor 3 \delta n /2 \rfloor \}$ as desired.

\medskip \noindent
\textbf{Case 2: $|M_R| + |S| < \eps' n + p /2$.}
Together with \eqref{eqn:M_R} we have $|M_R| < 2p/3$ and $p >0$.
Thus $|Y| = \eps' n - p$.

\medskip \noindent
\textbf{Case 2a: $|S \cap Y| \le \eps' n - 10 p /3 $ }.
Note that by~\eqref{eqn:Y}
\begin{align*}
 | Y \setminus (S \cup V(M_R)) | \ge |Y| - |S\cap Y| - 2 |M_R| 
 \ge \eps' n - p - ( \eps' n - 10 p /3) - 4p/3 = p .
\end{align*}

By~\eqref{eqn:M_R}, $|M_R| + |S| \ge \eps' n $.
We can extend $M_R$ into a matching $M$ in~$F= G\setminus A$ such that $|M| = \lceil  \eps' n \rceil $ and $|Y \setminus V(M)| \ge p$.
Indeed this is possible, by adding appropriate edges between~$S$ and $V(F) \setminus Y$ as $d_F(s) \ge 10 \eps' n \ge |Y|+9\eps' n$ for all $s \in S$.
Hence
\begin{align}
	& \left\lfloor \frac{3|A|}2 + |M| + |Y| -\frac{| V(M) \cap Y |}2 \right\rfloor
	= \left\lfloor \frac{3|A|}2 + |M| + \frac{|Y| + | Y \setminus V(M_R) |}2 \right\rfloor
	\nonumber \\
	&  \overset{ {\eqref{eqn:e2}, \eqref{eqn:Y}} }{\ge} \left\lfloor \frac{3 (\delta - \eps') n }2 + \eps' n  +  \frac{ (\eps' n - p) + p }2 \right\rfloor 
	= \left\lfloor \frac{3 \delta n }2 \right\rfloor.  \nonumber
\end{align}
We are done by Lemma~\ref{lma:matching1} (with $M,\emptyset,\eps'$ playing the roles of $M,M', \eps$).

\medskip \noindent
\textbf{Case 2b: $|S \cap Y| > \eps' n - 10 p /3 $.}
Recall that $|  M_R | < 2p/3$ and $|M_R|+ |S| \le \eps' n + p/2$.
So 
\begin{align}
| (S \cup V(M_R) ) \cap (B \cup Z) | & = | (S \cup V(M_R) ) \setminus Y | 
\le | S |+ 2|M_R| -  |S \cap Y|  \nonumber \\
& \le \eps' n + p/2 + 2p/3 - (\eps' n - 10 p /3 )
= 9 p /2 . \label{eqn:SM_R}
\end{align}
Let $F'$ be the subgraph $G[A,B \cup Z]$ obtained by removing all edges $uv$ with $c(uv) = c_a$ for some $a \in A$.
Note that for each $a \in A$, 
\begin{align*}
d_{F'}(a) \ge \delta^c(G) - ( 1 + |V(G) \setminus (B \cup Z)| - 1 ) 
= \delta n - |A| - |Y| 
= \eps' n - |Y| = p.
\end{align*}
Hence, $e(F') \ge p |A| \ge p(\delta- \eps') n$ and $\Delta (F') \le 24 \eps' n$ as $(A, B \cup Z)$ is $24 \eps'$-extremal by Proposition~\ref{prop:BZ}.
Since $\eps' \ll \delta$, K\"{o}nig's theorem implies that there is a matching 
\begin{align*}
 e(F')/ \Delta (F') \ge  11 p/2 \overset{\eqref{eqn:SM_R}}{\ge} p + | (S \cup V(M_R) ) \cap V(F') | .
\end{align*}
Thus there is a matching $M'$ in $F' \subseteq G[A,B \cup Z]$ such that $ | M'| =  p $ and $V(M') \cap (V (M_R) \cup S ) = \emptyset$.
By adding (appropriate) edges of~$F$ incident with~$S$, we can extend $M_R$ into a matching~$M$ in~$F= G\setminus A$ satisfying $V(M) \cap V(M') = \emptyset$, $|M| = \lceil \eps' n \rceil$.
Note that $|M|+|M'| = p +\lceil \eps' n \rceil \le 2 \eps'n+1$ and 
\begin{align}
	 &\left\lfloor \frac{3|A|}2 + |M|  + \frac{|M'|}2 + |Y| -\frac{| V(M) \cap Y |}2 \right\rfloor 
	\ge \left\lfloor \frac{3|A|}2 + |M|  + \frac{|M'|}2 + \frac{|Y |}2 \right\rfloor
	\nonumber \\
	&\overset{ {\eqref{eqn:e2}, \eqref{eqn:Y}} }{\ge}  \left\lfloor \frac{3 (\delta - \eps') n }2 + \eps' n + \frac{p}2  +  \frac{\eps' n - p}2 \right\rfloor 
	= \left\lfloor \frac{3 \delta n }2 \right\rfloor.  \nonumber
\end{align}
Again, we are done by Lemma~\ref{lma:matching1} (with $M,M',2\eps'$ playing the roles of $M,M', \eps$).
\end{proof}

\section{Absorbing cycles} \label{sec:abs}

In this section, we prove Lemma~\ref{lma:abscycle}.
We need the following definitions. 
Given a vertex $x$, we say that a path $P$ is an \emph{absorbing path for~$x$} if the following conditions hold:
\begin{enumerate}
	\item[\rm (i)] $P = z_1z_2z_3z_4$ is a properly coloured path of length~$3$; 
	\item[\rm (ii)] $ x \notin V(P)$;
	\item[\rm (iii)] $z_1 z_2 x z_3 z_4$ is a properly coloured path.
\end{enumerate}
Next we define an absorbing path for two disjoint edges.
Given two vertex-disjoint edges $x_1x_2$, $y_1y_2$, we say that a path $P$ is an \emph{absorbing path for $(x_1, x_2; y_1, y_2)$} if the following conditions hold:
\begin{enumerate}
	\item[\rm (i)] $P = z_1z_2z_3z_4$ is a properly coloured path of length~$3$;
	\item[\rm (ii)] $V(P) \cap  \{ x_1, x_2, y_1, y_2\} = \emptyset$; 
	\item[\rm (iii)] both $z_1 z_2 x_1 x_2 $ and $y_1 y_2 z_3 z_4$ are properly coloured paths of length~3.
\end{enumerate}
Note that the ordering of $(x_1, x_2; y_1, y_2)$ is important.
We would also need the following proposition from~\cite{LoDirac}.

\begin{proposition} \label{prp:abspath}
Let $P' = x_1 x_2 \dots x_{\ell-1} x_{\ell}$ be a properly coloured path with $\ell \ge 4$. 
Let $P = z_1 z_2 z_3 z_4$ be an absorbing path for $(x_1, x_2; x_{\ell-1}, x_{\ell} )$ with $V(P) \cap V(P') = \emptyset$.
Then $z_1 z_2 x_1 x_2 \dots x_{\ell -1} x_{\ell} z_3z_4$ is a properly coloured path.
\end{proposition}

Given a vertex $x$, let $\mathcal{L}(x)$ be the set of absorbing paths for $x$.
Similarly, given two vertex-disjoint edges $x_1x_2$, $y_1y_2$, let $\mathcal{L}(x_1, x_2; y_1,y_2)$ be the set of absorbing paths for $(x_1, x_2; y_1,y_2)$.
The following lemma follows immediately from Lemmas~4.3 and~4.5 of~\cite{LoDirac}.

\begin{lemma} \label{lma:absorbing}
Let $0 < 1/n \ll \gamma \ll \eps < 1/2$. 
Let $G$ be an edge-coloured graph on $n$ vertices with $\delta^c(G) \ge (1/2 + \eps)n$.
Then there exists a family $\mathcal{F}$ of vertex-disjoint properly coloured paths each of length~$3$, which satisfies the following properties:
\begin{align*}
|\mathcal{F}| & \le  \gamma^{1/2} n, &
|\mathcal{L}(x) \cap \mathcal{F}| & \ge \gamma n,  &
|\mathcal{L}(x_1, x_2;y_1, y_2) \cap \mathcal{F}| & \ge  \gamma n
\end{align*}
for all $x \in V(G)$ and for all distinct vertices $x_1,x_2, y_1,y_2 \in V(G)$ with $x_1x_2$, $y_1y_2 \in E(G)$.
\end{lemma}

To prove Lemma~\ref{lma:abscycle}, we aim to join the paths in $\mathcal{F}$ given by Lemma~\ref{lma:absorbing} into a properly coloured cycle. 
First, we need the following definition, which are only used in this section.

Let $G$ be an edge-coloured graph on $n$ vertices.
Let $x, y \in V(G)$ be distinct and let $\ell  \in \mathbb{N}$.
Define $\mathcal{P}^G_{\ell}(x; y)$ to be the set of properly coloured paths~$P$ of length~$\ell$ from $x$ to~$y$.
Define $	\mu^G _{\ell}(x; y) : = |\mathcal{P}^G_{\ell}(x; y)| / n^{\ell - 1}$ and $	\mu ^G _{ \le \ell}(x; y) : = \sum_{\ell' \le \ell} \mu ^G _{\ell'}(x; y)$.
For a colour set $C_y$, let $\mathcal{P}^G_{\ell}(x ; y, C_y)$ be the set of paths $P \in \mathcal{P}_{\ell}(x;y)$ such that $C_{P}(y) \in C_y$.
Define $\mu ^G _{ \ell } (x ; y, C_y)$ and $\mu ^G _{\le  \ell } (x ; y ,C_y)$ analogously. 
For $\ell \in \mathbb{N}$ and $\eta >0$, we say that $y$ is \emph{$(\le \ell,\eta)$-reachable from~$x$ in $G$} if $\mu ^G _{ \le \ell } (x ; y) \ge \eta$.
We say that $y$ is \emph{strongly $(\le \ell,\eta)$-reachable from~$x$ in $G$} if for any colour $c_0$, $y$ is $(\le \ell,\eta)$-reachable from $x$ in $G  - \{yz \in E(G): c(yz) = c_0\}$.
Equivalently, $y$ is strongly $( \le  \ell,\eta)$-reachable from~$x$ in~$G$ if $\mu ^G _{  \le \ell } (x ; y ,C(G) \setminus c_0)  \ge \eta $ for all colours~$c_0 \in C(G)$.

\begin{proposition} \label{prop:reachable}
Let $\ell \in \mathbb{N}$ and let $\eta >0$.
Let $G$ be an edge-coloured graph on $n$ vertices.
Let $x,y,v$ be distinct vertices in $V(G)$.
\begin{enumerate}[label={\rm(\roman*)}]
 	\item If $y$ is strongly $(\le  \ell,\eta)$-reachable from~$x$, then for any colour $c_0$, we have $\mu^{G \setminus v} _{ \le \ell}(x ; y, C(G) \setminus c_0) \ge \eta - \ell^2/n $.
\end{enumerate} 
If $y$ is not strongly $(\le  \ell,\eta)$-reachable from~$x$ but is $(\le  \ell, 2 \eta)$-reachable, then 
 \begin{enumerate}[label={\rm(\roman*)}, resume]
	\item there exists a unique colour $c_{y}$ such that $ \mu^{G} _{\le \ell}(x; y, c_{y} ) \ge \eta$;
	\item $ \mu^{G \setminus v} _{ \le \ell} (x ; y,c_y)  \ge \eta - \ell^2 / n$.
\end{enumerate}

\end{proposition}

\begin{proof}
For each $\ell' \in \mathbb{N}$, $v$ is in at most $(\ell'-1) n^{\ell'-2}$ paths of length $\ell'$ from $x$ to~$y$.
Hence for all $\ell' \le \ell$,
\begin{align*}
\mu^{G \setminus v} _{\ell'}(x ; y, C(G) \setminus c_0) 
& \ge \mu^{G} _{\ell'}(x ; y , C(G) \setminus c_0) - (\ell'-1)/n\\
& \ge  \mu^{G} _{\ell'}(x ;  y ,C(G) \setminus c_0 ) - \ell/n,
\end{align*}
so (i) holds.
The definitions of $(\le  \ell,2 \eta)$-reachable and strongly $(\le  \ell,\eta)$-reachable implying~(ii).
The proof of~(i) can be adapted to prove (iii).
\end{proof}

\begin{lemma} \label{lma:connect}
Let $0 < 1/n \ll \eps < 1/2$.
Suppose that $G$ is an edge-coloured graph on $n$ vertices with $\delta^c (G) \ge ( 1 / 2 + \eps) n +2$.
Let $x,y \in V(G)$ be distinct and let $c_x, c_y$ be any two colours.
Then there exists a properly coloured path $P$ from $x$ to $y$ of length at most $\eps^{-2}$ such that $C_P(x) \ne \{c_x\}$ and $C_P(y) \ne \{c_y\}$.
\end{lemma}

\begin{proof}
Let $\ell_0 : = \lfloor \eps^{-2} \rfloor$ and let $\eta$ be such that $1/n \ll \eta \ll \eps $.
Let $G, x,y , c_x, c_y$ be as defined in the lemma.
Remove all edges at $x$ with colour~$c_x$ and all edges at $y$ with colour~$c_y$.
So $d(x), d(y) \ge (1/2+\eps)n $ and $ d^c(v) \ge (1/2+\eps)n$ for all $v \in V(G) \setminus \{ x , y \}$.
Therefore to prove the lemma, it suffices to show that there exists a properly coloured path from $x$ to $y$ of length at most $\ell_0$.
Note that for all $v \in V(G)$, all $\ell \le \ell_0$ and all $P \in P_{\ell}^G(x;v)$, we may assume that $y \notin V(P)$ or else the lemma holds.

For each $\ell \in \mathbb{N}$, let $S_{\ell}$ be the set of vertices $v \in V(G) \setminus  x$ that are strongly $(\le \ell, \eta^{\ell})$-reachable from~$x$, and let $T_{\ell}$ be the set of vertices $v \in V(G) \setminus ( S_{\ell} \cup x ) $ that are $( \le \ell, 2 \eta^{\ell})$-reachable from~$x$.
Since a $(\le \ell, 2 \eta^{\ell})$-reachable vertex from~$x$ is also $(\le \ell+1, 2 \eta^{\ell+1})$-reachable from~$x$ and a similar statement for strongly reachable, we have
\begin{align}
S_{\ell} \subseteq S_{\ell+1} \text{ and } S_{\ell} \cup T_{\ell} \subseteq S_{\ell+1} \cup T_{\ell +1} \text{ for all }\ell \in \mathbb{N}. \label{eqn:subset}
\end{align}
Also $S_1 = \emptyset$  and $T_1$ is the set of vertex $v \in N(x)$, so
\begin{align}
|T_1| \ge ( 1 /2 + \eps) n. \label{eqn:T1}
\end{align}
Suppose that there exists $s \in S_{\ell} \cap N(y)$.
Let $P \in \mathcal{P}^{G}_{\ell}(x ;  s)$ with $c(sy) \notin C_P(s)$ (which exists as $s$ is strongly $(\le  \ell,\eta)$-reachable from~$x$).
Note that $Py$ is a properly coloured path from $x$ to $y$ of length at most $\ell +1$.
Thus we may assume that $|S_{\ell}| \le (1/2- \eps ) n $ for all $\ell < \ell_0$.
If $2 |S_{\ell+1}| + |T_{\ell+1}| \ge 2 |S_{\ell}| + |T_{\ell}| +  \eps^2 n$ for all $1 \le \ell < \ell_0-1$, then together with~\eqref{eqn:T1} we have $2 |S_{\ell_0-1}| + |T_{\ell_0-1}| \ge 3n/2$.
Hence $|S_{\ell_0-1}|  \ge n/2$, a contradiction.
Therefore, we may assume that for some $\ell < \ell_0-1$,
\begin{align}
2 |S_{\ell+1}| + |T_{\ell+1}| < 2 |S_{\ell}| + |T_{\ell}| +    \eps^2 n.  \label{eqn:q1}
\end{align}
By~\eqref{eqn:subset}, we have
\begin{align}
|(S_{\ell+1}\cup T_{\ell+1}) \setminus (S_{\ell}\cup T_{\ell})| \le   \eps^2 n . \label{eqn:q2}
\end{align}
Let $W := T_{\ell} \cap T_{\ell+1}$.
Recall that $|S_{\ell}| \le  (1/2- \eps ) n$. 
By \eqref{eqn:subset} and~\eqref{eqn:T1}, we have 
\begin{align*}
	|T_{\ell}| \ge |S_{\ell } \cup T_{\ell}| - |S_{\ell}| \ge |T_1| -  (1/2- \eps ) n \ge2  \eps n. 
\end{align*}
Since $T_{\ell} \setminus W = T_{\ell} \setminus T_{\ell +1 } \subseteq S_{\ell+1} \setminus S_{\ell} \subseteq (S_{\ell+1} \cup T_{\ell+1} ) \setminus ( S_{\ell} \cup T_{\ell})$ by~\eqref{eqn:subset}, \eqref{eqn:q2} implies that
\begin{align}
	\label{eqn:TW1}
	|T_{\ell} \setminus W | \le  \eps^2 n 
\end{align} 
and so
\begin{align}
	\label{eqn:W2}
|W| \ge |T_{\ell}| - |T_{\ell} \setminus W| \ge 2 \eps n   -  \eps^2 n  \ge \eps n.
\end{align}

For each $w \in W \subseteq T_{\ell}$, Proposition~\ref{prop:reachable}(ii) implies that there exists a unique colour $c_{w}$ such that $ \mu^{G} _{ \le \ell }(x ; w, c_{w}) \ge \eta^{\ell}$.
Define an auxiliary digraph $H$ with on $V(G) \setminus x$ and edge set $E(H) := \{ wv : w \in W, v \in N_G(w) \setminus x$ and $c(wv) \ne c_w \}$.
Note that for each $w \in W$, we have $d_H^+(w) \ge d^c_G(w)-1 \ge ( 1+\eps )n/2$ and so 
\begin{align}
	e(H) \ge ( 1+\eps )n|W|/2. \label{eqn:E(H)}
\end{align}

We now bound $e(H)$ from above (to obtain a contradiction) in the following claim.
\begin{claim} \label{clm:e(H)}
Let $e_H(X,Y)$ denote the number of edges from $X$ to $Y$.
Then
	\begin{enumerate}[label={\rm(\roman*)}]		
		\item \label{itm:W1}
		$e_H(W, (S_{\ell+1}\cup T_{\ell+1}) \setminus (S_{\ell}\cup T_{\ell}) )	<  \eps^2 n |W|$;
		\item \label{itm:W5}
		$e_H( W , T_{\ell} \setminus  W)  < \eps^2 n |W|$;
		\item \label{itm:W2}
		$e_H(W,V(G) \setminus (S_{\ell+1} \cup T_{\ell+1} \cup x )) < 4 \eta \eps^{-1} n |W|$;
		\item \label{itm:W3}
		$e_H( W , S_{\ell} )  < 2 \eta n|W|$;
		\item \label{itm:W4}
		$e_H( W , W )  <  (1/2 - \eps + 2 \eta) n |W|$.
	\end{enumerate}
\end{claim}

\begin{proof}[Proof of claim]
Note that \ref{itm:W1} and~\ref{itm:W5} follow from \eqref{eqn:q2} and~\eqref{eqn:TW1}, respectively.
To see \ref{itm:W2}, note that if $wv \in E(H)$ with $w \in W$ and $v \in V(G) \setminus x$ and $P \in \mathcal{P}^{G \setminus v}_{\ell'}(x ; w,c_w)$, then $Pv$ is a properly coloured path of length~$\ell'+1$ from~$x$ to~$v$.
By Proposition~\ref{prop:reachable}(iii), for each $v \in V(G)\setminus x$,
\begin{align*}
	\mu^{G} _{\le \ell+1}(x, v) & \ge \frac1n \sum_{w \in N_H(v)}  \mu^{G \setminus x } _{\le \ell}(x ; w, c_w)
 \ge \eta^{\ell} e_H( W,v)/2n.
\end{align*}
Therefore, for all $v \in V(G) \setminus (  S_{\ell+1} \cup T_{\ell+1} \cup x )$, we have $e_H( W,v ) < 4 \eta n \le 4 \eta \eps^{-1} |W|$, where the last inequality is due to~\eqref{eqn:W2}.
Thus \ref{itm:W2} holds.

Consider the edge $ws \in E(H)$ with $w \in W$ and $s \in S_{\ell}$.
If $P \in \mathcal{P}^{G \setminus v}_{\ell'}(x ; s,C(G) \setminus c(ws))$, then $Pw$ is a properly coloured path of length $\ell' +1$ from $x$ to~$w$ with $C_{P}(w) \ne \{c_w\}$.
We must have $e_H(w,S_{\ell}) < 2\eta n$ for all $w \in W$, which in turn implies~\ref{itm:W3}.
Indeed, if $e_H(w , S_{\ell}) \ge 2\eta n$, then by Proposition~\ref{prop:reachable}(iii),
\begin{align*}
	\mu^{G} _{\le \ell+1}(x ; w, C(G) \setminus c_w) & \ge \frac1n  \sum_{s \in N_H(w) \cap S_{\ell}} \mu^{G \setminus v}_{\le \ell}(x ;  s,C(G) \setminus c(ws)) \\
& \ge \frac1n e_H(w, S_{\ell})(\eta^{\ell} - \ell^2 /n) \ge \eta^{\ell+1}
\end{align*}
and so $w \in S_{\ell+1}$ (as $w \in W \subseteq T_{\ell+1}$ implying that $\mu^{G} _{\le \ell+1}(x; w, c_w) \ge \eta^{\ell+1}$), a contradiction.

By a similar argument with ($T_{\ell}$  playing the role of $S_{\ell}$), we deduce that every $w \in W \subseteq T_{\ell+1}$ has less than $2 \eta n $ edges $w w'$ in~$G$ such that $w' \in W \subseteq T_{\ell}$ and $c_w \ne c(w w') \ne c_{w'}$.
This means that, in $H$, each $w \in W$ is contained less than $2 \eta n $ $2$-cycles. 
Since each $w \in W$ is incident to at most $(1/2-\eps)n$ edges of the same colour in~$G$,
we have $e_H(W,w) < (1/2-\eps)n+ 2 \eta n = (1/2 - \eps + 2 \eta) n$ implying~\ref{itm:W4}.
\end{proof}

By Claim~\ref{clm:e(H)}, we deduce that 
\begin{align*}
e(H) & 
 \le \left( \eps^2 +\eps^2 + 4 \eps^{-1} \eta + 2 \eta   + 1/2- \eps + 2\eta \right)n |W| 
 < (1+ \eps) n |W|/2,
\end{align*}
contradicting~\eqref{eqn:E(H)}.
This complete the proof of Lemma~\ref{lma:connect}.
\end{proof}

We now prove Lemma~\ref{lma:abscycle}.

\begin{proof}[Proof of Lemma~\ref{lma:abscycle}]
Let $\eps_0$ be such that $1/n \ll \eps_0  \ll \eps$.
Apply Lemma~\ref{lma:absorbing} and obtain a family $\mathcal{F}$ of vertex-disjoint properly coloured paths each of length~$3$ such that for all $x \in V(G)$ and for all distinct vertices $x_1,x_2, y_1,y_2 \in V(G)$ with $x_1x_2$, $y_1y_2 \in E(G)$,
\begin{align*}
|\mathcal{F}| & \le 3 \gamma^{1/2} n, &
|\mathcal{L}(x) \cap \mathcal{F}| & \ge 3 \gamma n,  &
|\mathcal{L}(x_1, x_2;y_1, y_2) \cap \mathcal{F}| & \ge  3 \gamma n.
\end{align*}
Let $P_1, \dots, P_{|\mathcal{F}|}$ be paths in $\mathcal{F}$.
Let $x_i$ and $y_i$ be endvertices of $P_i$ for all $i \le |\mathcal{F}|$.
Suppose that for $j \le |\mathcal{F}|$, we have already found $Q_1, \dots, Q_{j-1}$ such that
\begin{enumerate}[label={\rm (\alph*)}]
	\item for all $i < j$, $Q_i$ is a path from $y_i$ to $x_{i+1}$ of length at most $\eps_0^{-2}$;
	\item for all $i < j$, $P_{i} Q_{i} P_{i+1}$ is a properly coloured path;
	\item $Q_1, \dots, Q_{j-1}, P_{j+1}, \dots, P_{|\mathcal{F}|}$ are disjoint.
\end{enumerate}
We now find $Q_{j}$ as follows. 
Let $C_{P_j}(y_j) = \{c_y\}$, let $C_{P_{j+1}}( x_{j+1} ) = \{c_x\}$ and let $W : = ( \bigcup_{i \le |\mathcal{F}|}  V(P_i) \cup \bigcup_{i' <j} V(Q_{i'}) ) \setminus \{y_j,x_{j+1}\}$, where we take $P_{|\mathcal{F}|+1} = P_1$ and $x_{|\mathcal{F}|+1} = x_1$.
Note that $|W| \le 3 \gamma^{1/2} n (4+\eps^{-2}_0) \le \eps n/2$.
Let $G' = G \setminus W$.
So $\delta^c(G') \ge (1/2 + \eps/2) n \ge (1/2 + \eps_0)|G'|$.
Apply Lemma~\ref{lma:connect} and obtain a properly coloured path $Q_j$ in $G'$ from $y_{j}$ to $x_{j+1}$ of length at most $\eps_0^{-2}$ such that $C_{Q_j}(y_j) \ne \{c_y\}$ and $C_{Q_j}(x_{j+1}) \ne \{c_x\}$.
Thus we have found $Q_1, \dots, Q_{|\mathcal{F}|}$.

Let $C := P_1 Q_1 P_2 \dots P_{|\mathcal{F}|} Q_{ |\mathcal{F}| }$ be a properly coloured cycle in~$G$. 
Note that $|C| \le 3 \gamma^{1/2} n (4+\eps^{-2}_0) \le \eps n/2$.
Let $\mathcal{P}$ be any set of $k$ vertex-disjoint properly coloured paths in $G \setminus V(C)$ with $k \le \gamma n$.
Let $\mathcal{P'}$ be the set of properly coloured paths obtained from~$\mathcal{P}$ by breaking up every path $P \in \mathcal{P}$ with $|P| \le 3$ into isolated vertices. 
Thus $|\mathcal{P}'| \le 3 \gamma n$ and for each $P \in \mathcal{P}'$, $|P| =1$ or $|P| \ge 4$.
For each $P \in \mathcal{P}'$, there exists a distinct $P' \in \mathcal{F}$ such that $P' \in \mathcal{L}(V(P))$ if $|P'| = 1$, and $P' \in \mathcal{L}(u_1,u_2;u_{\ell'} u_{\ell'-1})$ if $P = u_1 u_2 \dots u_{\ell'}$.
By Proposition~\ref{prp:abspath} and the definition of an absorbing path for a vertex, there exists a properly coloured cycle $C'$ with vertex set $V(C) \cup V (\bigcup \mathcal{P} )$.
\end{proof}

\section{Properly coloured 1-path-cycle}
\label{sec:1-path-cycle}

A \emph{$1$-path-cycle} is a disjoint union of cycles and at most one path.
In this section, we prove the following lemma, which immediately implies Lemma~\ref{lma:pathcovers}.

\begin{lemma} \label{lma:1-path-cycle}
Let $0 < 1/n \ll \beta \ll \eps \ll 1/2 < \delta$.
Suppose that $G$ is a critical edge-coloured graph on $n$ vertices with $\delta^c(G)  \ge \delta n+1 $.
Then one of the following statements holds
\begin{enumerate}[label={\rm(\roman*)}]
\item $G$ contains a properly coloured $1$-path-cycle $H$ such that $|H| \ge \min \{ ( 3 \delta + \beta )n/2, n \}$ and every cycle in $H$ has length at least $\beta n /100$;
\item $G$ is $(\delta, \eps)$-extremal.
\end{enumerate}
\end{lemma}

To prove Lemma~\ref{lma:1-path-cycle}, we need the following terminology. 
Let ${\bf x} = (x,c_x)$ and ${\bf y} = (y,c_y)$ be pairs with vertices $x,y \in V(H)$ and colours $c_x,c_y$.
For $\rho>0$, we say that $H$ is a \emph{$1$-path-cycle with parameters $\rho$-$( {\bf x} ; {\bf y} )$} if $H$ satisfies the following four properties:
\begin{enumerate}[label={\rm (\alph*)}]
	\item $H$ is a properly coloured $1$-path-cycle;
	\item every cycle in~$H$ has length at least $\rho n$;
	\item the path component~$P$ in $H$ has length at least $\rho n$ with endvertices $x$ and $y$;
	\item$C_H(x) = \{ c_{x} \}$ and $C_H(y) = \{c_{y} \}$.
\end{enumerate}
Note that $c_{x}$ and $c_{y}$ are precisely the colours of the edges in $P$ (and $H$) incident with $x$ and $y$, respectively.
The order of ${\bf x}$ and ${\bf y} $ is important.
If $\rho$ is known from the context, we simply write $( {\bf x} ; {\bf y} )$ instead of $\rho$-$( {\bf x} ; {\bf y} )$.

Orient the cycles of $H$ into directed cycles arbitrarily and orient the path $P$ into a directed path from $x$ to $y$.
For each $v \in V(H) \setminus  y$, define $c_+(v) $ to be $c( v v_+)$, where $v_+$ is the successor of~$v$, and for each $w \in V(H) \setminus x$, define $c_-(w) $ to be $c( w w_-)$, where $w_-$ is the ancestor of~$w$.
From now on every $1$-path cycle is assumed to be oriented as above.
For an oriented cycle $C$ and $u,v \in V(C)$, we write $uC^+v$ for the path $u u_+ \dots v_- v$ in~$C$ and $uC^-v$ for the path $u u_- \dots v_+ v$ in~$C$.

\begin{lemma} \label{lma:parameter}
Let $\rho > 0$.
Let $G$ be an edge-coloured graph on $n$ vertices with $\delta^c (G) \ge \rho n+1$.
Suppose that $H$ is a properly coloured $1$-path-cycle in $G$ of maximum order such that every cycles in $H$ has length at least $\rho n$, and that $|H|< n$.
Then there exists a $1$-path-cycle $H'$ with parameters $\rho$-$( {\bf x} ; {\bf y} )$ such that $V(H') = V(H)$.
\end{lemma}

\begin{proof}
If $H$ contains no path component, then $H + w$ is a properly coloured $1$-path-cycle such that every cycle has length at least $\rho n $, where $w \in V(G) \setminus V(H)$.
This contradicts the maximality of $|H|$. 
So we may assume that $H$ contains a path component~$P$.

Suppose that $P$ has length less than~$\rho n$. 
Let $x$ be an endvertex of~$P$.
Let ${\bf x} = (x,c_x)$ with $C_{P}(x) = \{ c_{x} \}$ if $|V(P)| \ge 2$, and $c_x$ is an arbitrary colour otherwise.
Note that $ | N ( {\bf x} ) | \ge  \delta^c (G) -1\ge \rho n \ge |V(P) \setminus x| $.
So there exists $w \in N ( {\bf x} ) \setminus V(P)$.
If $w \notin V(H)$, then we can extend~$P$ thus enlarging~$H$, a contradiction.
Hence $w \in V(H) \setminus V(P)$ and let $C$ be the cycle in~$H$ containing~$w$.
Without loss of generality, we may assume that $c(x w) \ne c_-(w)$.
Then $H' = H + xw - ww_-$ is a properly coloured $1$-path-cycles on vertex set~$V(H)$ such that every cycle in~$H$ has length at least~$\rho n $ and the path component is $P' = w_+ C^+ w x P$ of length at least~$|C| \ge \rho n $.
Therefore $H'$ is a $1$-path-cycles with parameters $( {\bf w_+ }  ; {\bf y } )$, where ${\bf w_+ } = (w_+ , c_{+}(w_+) )$ and ${\bf y} = (y, c_y)$ such that $y$ is the other endvertex of $P'$ and $C_{P'}(y) = \{ c_{y} \}$.
\end{proof}

In the next proposition, we show how we can change from $1$-path-cycle to another one by `switching edges'.

\begin{proposition} \label{rotation}
Let $G$ be an edge-coloured graph.
Let $\rho >0$.
Let $H$ be a $1$-path-cycle in $G$ with parameters $( {\bf x } ; {\bf y } )$, where ${\bf x } = (x,c_x)$ and ${\bf y } = (y,c_y)$.
Suppose that $w \in V(H) \cup N_G( {\bf x} )$ such that  $\dist_H(w,x), \dist_H(w,y) \ge \rho n +1$.
Then 
\begin{enumerate}[label={\rm(\roman*)}]
\item if $c(xw) \ne c_- (w)$, then $H + x w - w w_+$ is a $1$-path-cycle with parameters $( ( w_+,c_+(w_+) ) ; {\bf y})$;
\item if $ c ( x w ) \ne c_+ (w )$, then $H + x w - w w_-$ is a $1$-path-cycle with parameters $( ( w_- , c_-(w_-)) ;  {\bf y})$.
\end{enumerate}
A similar statement holds for $w \in V(H) \cup N_G( {\bf y} )$ with $\dist_H(w,x), \dist_H(w,y) \ge \rho n +1$.
\end{proposition}

\begin{proof}
Suppose that $c(xw) \ne c_- (w)$.
If $w$ is in the path component~$P$ of~$H$, then $P + x w - w w_+$ is a properly coloured graph consisting of a cycle $x P w x$ and a path $w_+ P y$ (as $c(xw) \ne c_x$).
Since $\dist_H (w,x), \dist_H (w,y) \ge \rho n +1$, both of these components have size at least $\rho n $. 
Thus $H + x w - w w_+$ is a $1$-path-cycle with parameters $( ( w_+,c_+(w_+) ) ; {\bf y} )$.
If $C$ is the cycle in $H$ containing~$w$, then $P + C + x w - w w_+$ is a properly coloured path $w_+ C_+ w x P y$.
Hence $H + x w - w w_+$ is a $1$-path-cycle with parameters $( ( w_- , c_-(w_-)) ;  {\bf y})$.
Therefore (i) holds, and (ii) holds by a similar argument.
\end{proof}

Let $H$ be $1$-path-cycle in~$G$ with parameters $({\bf x} ; {\bf y})$ and let $H'$ be an  $1$-path-cycle with parameters $({\bf z} ; {\bf y})$ in~$G$ obtained from $H$ by switching one edges. 
Note that we can deduce which edges were involved in the switching by analysing ${\bf z}$ as follows. 
Let ${\bf z} = (z,c_z)$ be a pair with vertex $z \in V(H) \setminus \{x,y\}$ and colour $c_z \in C_H(z)$.
Define the vertex
\begin{align*}
w_{\bf z} : = 
\begin{cases}
z_- & \text{if $c_z = c_+(z)$,}\\
z_+ & \text{if $c_z = c_-(z)$.}
\end{cases}
\end{align*}
Note that $H' = H + x  w_{\bf z} - w_{\bf z} z$ by Proposition~\ref{rotation}.

Let $X_1(H)$ be the set of pairs ${\bf z} = (z,c_z)$ with vertex $z \in V(H)$ and colour $c_z \in C_H(z)$ such that 
\begin{itemize}
		\item $H + x  w_{\bf z} - w_{\bf z} z$ is a $1$-path-cycle with parameters $( {\bf z} ; {\bf y})$;
		\item $\dist_H( w_{\bf z}, x), \dist_H( w_{\bf z}, y) \ge  2\rho n$.
\end{itemize}
Note that $\{ ({\bf z} ;{\bf y} ) : {\bf z}  \in X_1(H)\}$ is a subset of possible parameters of the $1$-path-cycle that can be obtained from $H$ by switching one edge of~$H$ with an edge incident to~$x$. 
We obtain the following properties of $X_1(H)$.

\begin{proposition} \label{prop:X1}
Let $G$ be an edge-coloured graph on $n$ vertices and let $\rho >0$.
Suppose that $H$ is a properly coloured $1$-path-cycle in $G$ of maximum order, and that $H$ has parameters $\rho-( {\bf x}  ; {\bf y} )$.
Let $z \in  N_G ( {\bf x} ) $ such that $\dist_H(z , x), \dist_H( z , y) \ge 2 \rho n +1$.
Then the following statements hold
\begin{enumerate}[label={\rm(\alph*)}]
		\item $N_G ( {\bf x}) \subseteq V(H)$;
		\item if $c(xz ) \ne c_-(z)$, then $(z_+,c_+(z_+)) \in X_1(H)$;
		\item if $c(xz) \ne c_+(z)$, then $(z_-,c_-(z_-)) \in X_1(H)$;
		\item for ${\bf z} \in X_1(H)$, $N_G ( {\bf z} ) \subseteq V(H)$.
\end{enumerate}
\end{proposition}

\begin{proof}
If $z \in N_G ( {\bf x} ) \setminus V(H)$, then $H+xz$ is a $1$-path-cycle with parameters $(z,c(xz) ; {\bf y})$ contradicting the maximality of $H$. 
Thus (a) holds, and (d) is proved similarly (by considering $H + x  w_{\bf z} - w_{\bf z}z$ instead of $H$).

If $c(x z ) \ne c_-(z)$, then $H + xz - z z_+$ is a $1$-path-cycle with parameters $( (z_+,c_+(z_+) ) ; {\bf y})$ by Proposition~\ref{rotation}(i).
So $(z_+,c_+(z_+)) \in X_1(H)$ implying~(b).
A similar argument shows that (c) holds. 
\end{proof}

We would also need to consider the set of $1$-path-cycles with parameters $({\bf z};{\bf y})$ that can be obtained from $H$ by replacing two edges of~$H$.
We now define $X_2$, which is the analogue of $X_1$ for replacing two edges of $H$ (with some additional constraints).
Let $X_2(H)$ be the set of pairs ${\bf z} = (z,c_{z})$ with vertex $z \in V(H)$ and colour $c_{z} \in C_H(z)$ such that there exist at least $10 \rho n $ pairs ${\bf z'} = (z',c_{z'}) \in X_1(H)$ satisfying 
\begin{itemize}
	\item $\dist _H(z,x), \dist _H(z,y), \dist _H(z',z) \ge 2 \rho n $ and 
	\item $H + x  w_{\bf z'} + z' w_{\bf z} - z w_{\bf z}  - z' w_{\bf z'}$ is a $1$-path-cycle with parameters $( {\bf z} ; {\bf y})$.
\end{itemize}

In the next lemma, we show that if $|X_1(H) \cup X_2(H)|$ is bounded above, then there exist disjoint $W^*,Z^* \subseteq V(G)$ such that $G[W^* \cup Z^*]$ is extremal with partition $W^*,Z^*$.
The proof relies on analysing the structure of $X_1(H)$, $X_2(H)$ and $N ( { \bf z})$ for ${\bf z} \in X_1(H)$.

\begin{lemma} \label{lma:X1X2}
Let $0 <1/n \ll  \rho \le \alpha /1000 < 1/1000$ and let $1/2 + 3 \alpha < \delta \le 2/3$.
Let $G$ be a critical edge-coloured graph on $n$ vertices with $\delta^c (G) \ge \delta n+1$.
Suppose that $H$ is a properly coloured $1$-path-cycle in $G$ of maximum order.
Suppose that $H$ has parameters $( {\bf x}; {\bf y})$, that $|X_1(H) \cup X_2(H)| \le (\delta + \alpha) n $ and that $|H| < n $.
Then there exist disjoint $W^*, Z^* \subseteq V(H)$ such that 
\begin{enumerate}[label={\rm(\roman*)}]
	\item $ |W^*| \ge (\delta - 7 \sqrt{\alpha} ) n $ and $|Z^*| \ge ( 2 \delta - 1  - 3 \alpha^{1/4} ) n$;
	\item for each $ w \in W^*$, there exists a distinct colour $c_w^*$ such that there are at least $|Z^*| - 3\sqrt{\alpha} n $ vertices $z \in Z^* \cap N_G(w)$ such that $c(zw) = c_w^*$;
	\item for each $z \in Z^*$, $d_G(z) \le ( \delta + 4 \alpha^{1/4} ) n $ and there are at least $( \delta - 6 \alpha^{1/4} ) n$ vertices $w \in W^* \cap N_G(z)$ and $c(zw) = c_w^*$.
\end{enumerate}
\end{lemma}

\begin{proof}
Write $X_1$ for $X_1(H)$ and $X_2$ for $X_2(H)$.
Let $Z$ be the set of vertices $z \in V(H)$ such that $\dist_H( z , x), \dist_H( z , y) \ge 2  \rho n$ and  
\begin{itemize}
\item[($\ast$)] there exists a colour $c_z \in C_H(z)$ such that $\textbf{z} = (z,c_z) \in X_1$ with $c( z w_{ \bf{z} }) = c ( x w_{ \bf{z} } )$. 
\end{itemize}
Let $Z'$ be the set of vertices $z \in Z$ such that both colours $c_z \in C_H(z)$ satisfy~($\ast$).
Clearly $Z' \subseteq Z$.

We now bound the sizes of $Z$ and $Z'$ from below.
\begin{claim} \label{clm:Z+Z'}
$|Z| + |Z'| \ge (\delta - 2 \alpha)n \ge n/2$.
\end{claim}

\begin{proof}[Proof of claim]
Let
\begin{align*}
	N &:= \{ u \in N_G( {\bf x} ) \colon \dist_H( u , x), \dist_H( u, y) > 2  \rho n \}, &
	N' &:= \{ u \in N \colon c(x u) \in C_H(u)  \}.
\end{align*}
Thus $|N| \ge \delta^c(G) - 1 - 2 \cdot 2 \rho n \ge ( \delta  - 4 \rho) n$ and $N \subseteq V(H)$ by Proposition~\ref{prop:X1}(a).
By Proposition~\ref{prop:X1}(b) and~(c),
\begin{align*}	
	|X_1| & \ge |N'| + 2 |N \setminus N'| = |N| + |N \setminus N'| 
 \ge (\delta  - 4 \rho ) n +|N \setminus N'|.
\end{align*}
Since $|X_1 \cup X_2| \le  (\delta  + \alpha) n$, we have $|N \setminus N'| \le (4 \rho + \alpha ) n $ and so 
\begin{align}
|N'| \ge |N| - |N \setminus N'| \ge (\delta - \alpha - 8 \rho )n \ge (\delta - 2 \alpha ) n . \nonumber 
\end{align}
Let $X'_1$ be the subset of~$X_1$ generated by the edges~$xv$ with~$v \in N'$, that is, $X'_1 := \{ (x', c_{x'}) \in X_1 : w_{(x', c_{x'})} \in N' \}$.
So $|X'_1| \ge (\delta  - 2\alpha  ) n $.
Thus if $(z,c_z) \in X'_1$, then $w_{ \bf{z} } \in N'$ and $c( z w_{ \bf{z} }) = c ( x w_{ \bf{z} } )$.
Note that $Z$ contains all vertices $z \in V(H)$ such that $(z,c_z) \in X'_1$ for some colour~$c_z$.
Similarly, $Z'$ contains all vertices $z \in V(H)$ such that $( z, c_+(z) ), (z, c_-(z) ) \in X'_1$.
Hence, $|Z| + |Z'| \ge |X'_1| \ge ( \delta - 2 \alpha )  n \ge n/2$ as required.
\end{proof}

Define a directed graph~$F$ on $V(H)$ such that there exists a directed edge from $z$ to $w$ if and only if 
\begin{itemize}
	\item $(z,c_z) \in X_1$ and $z \in Z \cap N_H( w)$ and $c(wz) \ne c_z$;
	\item $\dist_H (w,x), \dist_H (w,y), \dist _H (w,z) \ge 2 \rho n $.	
\end{itemize}
We also colour the edges~$uv$ (in $F$) by~$c(uv)$.
We now establish some properties of $F$.
 
\begin{claim} \label{clm:F}
~
\begin{enumerate}[label={\rm(\alph*)}]
\item $e(F)  \ge e_F(Z,V(F)) \ge
 (\delta  - 6 \rho) n |Z| + \sum_{z \in Z'} (d_G(z) - \delta n ) $.
\item If $w \in V(H)$ has $10 \rho n$ edges~$zw$ in~$F$ with $c(zw) \ne c_+(w)$, then $( w_-, c_-(w_-) ) \in X_2$.
\item If $w \in V(H)$ has $10 \rho n$ edges~$zw$ in~$F$ with $c(zw) \ne c_-(w)$, then $( w_+, c_+(w_+) ) \in X_2$.
\end{enumerate}
\end{claim}

\begin{proof}[Proof of claim]
For ${\bf z} \in X_1$, $N_{G} ( {\bf z} ) \subseteq V(H)$ by Proposition~\ref{prop:X1}(d). 
Hence, for each $z \in Z$, $d^+_F(z) \ge |N_{G} ( {\bf z} )| - 3 \cdot 2 \rho n   \ge ( \delta - 6 \rho) n$.
A similar argument implies that, for each $z \in Z'$, $d^+_F(z')  \ge d_G(z')- 6 \rho n$.
Hence (a) holds. 

Suppose that $zw$ is an edge in~$F$ with $c(zw) \ne c_+(w)$.
Thus there is ${\bf z} = (z,c_z) \in X_1$ such that $c_z \ne c(zw)$.
Note that by the definition of~$X_1$, $H' = H + x  w_{\bf z} - w_{\bf z} z$ is a $1$-path-cycle with parameters $( {\bf z} ; {\bf y})$.
Since $\dist _H (w,x), \dist_H (w,y), \dist _H (w,z) \ge 2 \rho n $, we have $\dist _{H'} (w,z), \dist _{H'} (w,y) \ge \rho n +1$.
Proposition~\ref{rotation}(ii) implies that $H' + zw - ww_-$ is a $1$-path-cycle with parameters $( ( w_-, c_-(w_-) ) ; {\bf y})$.
This implies (b), and (c) is proven similarly.
\end{proof}

Let $W := \{ w \in V(F) \colon d^-_F(w) \ge 20 \rho n \}$ and $W' := \{ w \in V(F) \colon d^-_F(w) \ge (1- 2 \sqrt{\alpha})|Z|\}$.
Let $W^*$ be the set of $w \in W'$ such that there exists a colour $c_w^*$ and there are at most $10 \rho n $ vertices $z \in N_G(w)$ with $c(zw) \ne c_w^*$.

\begin{claim} \label{clm:W^*}
$| W^*| \ge (\delta - 7 \sqrt{\alpha} ) n$, $|W \setminus W^*| \le 5 \sqrt{\alpha} n $  and 
\begin{align}
\frac{1}{n}\sum_{z \in Z'} (d_G(z) - \delta n)  + |W' \setminus W^* | 
\le 4 \sqrt{\alpha} n. \label{eqn:eF1}
\end{align}
\end{claim}

\begin{proof}[Proof of claim]
If $|W \setminus W'| > \sqrt{\alpha } n$, then Claim~\ref{clm:F}(a) implies that 
\begin{align*}
	( \delta  - 6 \rho) n |Z| & \le e_F(Z,V(F))
	\le  e_F(Z, W ) + 20 \rho n^2 
		\le |Z| |W|  - 2 \sqrt{\alpha}|Z| |W \setminus W'| + 20 \rho n^2  \\
		 & \le |Z||W| - 2\alpha |Z| n  + 20 \rho n^2  
		 \le  ( |W| - 2\alpha  n  + 80 \rho  n ) |Z|,
\end{align*}
where the last inequality holds as $|Z| \ge n/4$ by Claim~\ref{clm:Z+Z'}.
This implies that $|W| > ( \delta  + \alpha )n$.
By Claim~\ref{clm:F}(b) and~(c), we have $|X_2| \ge |W|$, a contradiction.
Hence,
\begin{align*}
	|W \setminus W'| \le \sqrt{\alpha } n.
\end{align*}
Thus we have
\begin{align}
e_F(Z,V(F)) & \le  e_F(Z, W) + 20 \rho n^2 \le  ( |W'| + ( \sqrt{\alpha}  + 80 \rho ) n ) |Z| 
 \le (|W'| + 2 \sqrt{\alpha} n) |Z| . 
\nonumber
\end{align}
By Claim~\ref{clm:F}(a), we have 
\begin{align}
|W'|	& \ge  (\delta  -  2 \sqrt{\alpha}  -  6 \rho) n  + \frac{1}{|Z|}\sum_{z \in Z'} (d_G(z) - \delta n) 
\nonumber \\ &
 \ge   (\delta - 3 \sqrt{\alpha} ) n  + \frac{1}{n}\sum_{z \in Z'} (d_G(z) - \delta n). \label{eqn:W'}
\end{align}
Note that if $w \in W' \setminus W^*$, then $(w_-,c_-(w_-) ) , ( w_+, c_+(w_+) ) \in X_2$ by Claim~\ref{clm:F}(b) and~(c).
Thus $|X_2| \ge |W'| + |W' \setminus W^* |$.
Since $|X_2| \le (\delta + \alpha) n$, \eqref{eqn:W'} implies that
\begin{align*}
\frac{1}{n}\sum_{z \in Z'} (d_G(z) - \delta n)  + |W' \setminus W^* | \le 
(\alpha + 3 \sqrt{\alpha}  ) n 
\le 4 \sqrt{\alpha} n,
\end{align*}
so \eqref{eqn:eF1} holds.
Moreover, $|W' \setminus W^* | \le 4 \sqrt{\alpha} n$, so $|W \setminus W^*| \le 5 \sqrt{\alpha} n$.
Together with~\eqref{eqn:W'}, $| W^*| = |W'| - |W' \setminus W^* |  \ge (\delta - 7 \sqrt{\alpha} ) n$.
\end{proof}

Recall that for each $w \in W^* \subseteq W'$, $d^-_F(w) \ge (1- 2 \sqrt{\alpha})|Z|$.
So for each $w \in W^*$, the number of edges $zw$ of colour $c_w^*$ in~$G$ is at least 
\begin{align}
	|\{ z \in N_G(w) \colon c(zw) = c_w^*\} | & \ge 
	(1- 2 \sqrt{\alpha})|Z| - 10 \rho n 
	 \ge |Z| - 3 \sqrt{\alpha} n .
	\label{eqn:W^*1}
\end{align}
Since $\delta^c (G) \ge  \delta n$, the left hand side of the inequality is bounded above by $(1-\delta)n$.
Thus $|Z| \le (1- \delta + 3 \sqrt{\alpha} ) n$ and so Claim~\ref{clm:Z+Z'} implies that
\begin{align}
|Z'| & \ge (2 \delta - 1  - 4 \sqrt{\alpha}) n. \label{eqn:Z'}
\end{align}

Let $Z^*$ be the set of vertices $z \in Z$ satisfying~(iii).
We now bound the size of $Z^*$ from below.

\begin{claim} \label{clm:Z^*}
$|Z^*| \ge (2 \delta -1- 3 \alpha^{1/4})n$.
\end{claim}

\begin{proof}[Proof of claim]
Let $Z_1$ be the set of $z \in Z'$ such that $d_G(z) \ge ( \delta + 4 \alpha^{1/4} ) n$.
So \eqref{eqn:eF1} implies that
\begin{align*}
|Z_1| \le \alpha^{1/4} n.
\end{align*}
Let $Z_2$ be the set of $z \in Z$ such that $d_G(z,V(F) \setminus W) \ge 20 \sqrt{ \rho} n$.
Note that
\begin{align*}
	|Z_2| \le e_F(Z, V(F) \setminus W) / 20 \sqrt{ \rho} n \le \sqrt{ \rho } n .
\end{align*}
Let $Z_3$ be the set of $z \in Z$ such that there exist at least $4 \alpha^{1/4} n $ vertices $w \in W^*$ with $c(zw) \ne c_w^*$.
By~\eqref{eqn:W^*1}, each $w \in W^*$ is incident with at most $3 \sqrt{\alpha} n $ edges $zw$ with $z \in Z$ and $c(zw) \ne c_w^*$.
Hence
\begin{align*}
	|Z_3| \le 3 \sqrt{\alpha} n ^2 / (4 \alpha^{1/4}n) < \alpha^{1/4} n. 
\end{align*}
For each $z \in Z \setminus (Z_2 \cup Z_3)$, the number of edges~$zw$ (in both $G$ and $F$) such that $w \in W^*$ and $c(zw) = c_w^*$ is at least
\begin{align*}
	d_G(z,W^*)  - 4 \alpha^{1/4} n \ge d_G(z) - 20 \sqrt{ \rho} n - |W \setminus W^*| - 4 \alpha^{1/4} n
	\ge ( \delta - 6 \alpha^{1/4} ) n,
\end{align*}
where the last inequality is due to Claim~\ref{clm:W^*}.
Hence $Z^*  \supseteq Z' \setminus (Z_1 \cup Z_2 \cup Z_3)$.
Together with~\eqref{eqn:Z'}, we have $|Z^*| \ge (2 \delta - 1 - 3 \alpha^{1/4} ) n$.
\end{proof}

Note that properties (i) and (ii) holds by Claims~\ref{clm:W^*} and~\ref{clm:Z^*} and~\eqref{eqn:W^*1}, and (iii) holds by our construction.
To complete the proof, it suffices to show that $W^*$ and $Z^*$ are disjoint.
For each $w \in W^*$, (ii) and (i) imply that 
\begin{align*}
	d_G(w) \ge d_G^c(w)-1  + |Z^*| -3 \sqrt{\alpha} n \ge (3\delta - 1 - 3 \alpha^{1/4}-3 \sqrt{\alpha}) n > ( \delta + 4 \alpha^{1/4} ) n,
\end{align*}
so $w \notin Z^*$ as required.
\end{proof}

Let $G$ be an edge-coloured graph and let $H$ be $1$-path-cycle with parameters $({\bf x} ; { \bf y} )$ with path component~$P$.
Let $H'$ be the $1$-path-cycle with parameters $({\bf y} ; { \bf x} )$ obtained from $H$ by reversing the orientations of all edges. 
Let $Y_1(H) : = X_1(H')$ and $Y_2(H) : = X_2(H')$.
We study the edges between $X_1(H) \cup X_2(H)$ and $Y_1(H) \cup Y_2(H)$ in the following lemma.

\begin{lemma} \label{lma:H}
Let $G$ be a critical edge-coloured graph on $n$ vertices and let $\rho >0$.
Suppose that $H$ is a properly coloured $1$-path-cycle in $G$ of maximum order. 
Suppose that  $H$ has parameters $({\bf x} ; { \bf y} )$ and that $|H| < n$.
Then for all $(x',c_{x'}) \in X_1(H) \cup X_2(H)$ and all $(y',c_{y'}) \in Y_1(H) \cup Y_2(H)$ such that $\dist_H(x,y) \ge 2 \rho n$, either $xy \notin E(G)$, $c(xy) = c_x$ or $c(xy) = c_y$.
\end{lemma}

\begin{proof}
Consider any ${\bf x'} =  (x',c_{x'}) \in X_1(H) \cup X_2(H)$ and any $ {\bf y'} =  (y',c_{y'}) \in Y_1(H) \cup Y_2(H)$ such that $\dist_H(x,y) \ge 2 \rho n$.
To prove the lemma, it is sufficient to show that there exists a $1$-path-cycle~$H_0$ with $V(H_0) = V(H)$ and parameters $( {\bf x'} ;  {\bf y'})$.
To see this suppose that $x'y' \in E(G)$ and $c_{x'}  \ne c(xy) \ne  c_{y'}$, then $H_0 + x'y'$ is a vertex-disjoint union of cycles each of length at least $\rho n$. 
For $z \notin V(H)$, $ ( H_0 + x'y' ) \cup z$ is a $1$-path-cycle contradicting the maximality of~$|H|$.

%
%
%

We will only consider the case when ${\bf x'} \in X_2(H)$ and ${\bf y'}  \in Y_2(H)$, since the other cases proved by similar (and simpler) arguments.
Choose ${\bf z} = (z,c_{z}) \in X_1(H)$ and ${\bf v} = ( v , c_{v}) \in Y_1(H)$
such that 
\begin{itemize}
	\item any pair of $\{x,y, x',y',z,v\}$ are distance at least $  \rho n+10$ apart in~$H$;
	\item $H' := H + x  w_{\bf z} + z w_{\bf x'} - z w_{\bf z}  - x' w_{ \bf x' }$ is a $1$-path-cycle with parameters $( { \bf x' } ; { \bf y})$.
	\item $H + y  w_{\bf v} + v w_{\bf y'} - v w_{\bf v}  - y' w_{\bf y'}$ is a $1$-path-cycle with parameters $( {\bf x} ; { \bf y'} )$.
\end{itemize}
Note that ${\bf z}$ and ${\bf v}$ exist since ${\bf x'} \in X_2(H)$ and ${\bf y'}  \in Y_2(H)$.
Since $\dist _H (v,x), \dist_H (v,y), \dist _H (v,z) \ge   \rho n +10$, we have $\dist _{H'} (v,x'), \dist _{H'} (v,y) \ge \rho n +1$.
Proposition~\ref{rotation} implies that $H'':= H' + y  w_{\bf v} - v w_{\bf v}$ is a $1$-path-cycle with parameters $( {\bf x'} ; {\bf v})$.
By a similar argument, we deduce that $H'' + v  w_{\bf y'} - y' w_{\bf y'}$ is a $1$-path-cycle with parameters $( {\bf x'} ; {\bf y'})$ as required. 
\end{proof}

The next lemma plays a key role in the proof of  Lemma~\ref{lma:1-path-cycle}.

\begin{lemma} \label{lma:AB}
Let $\eps , \rho, \alpha $ be such that $1/n \ll \alpha, \eps   \ll 1 $.
Let $G$ be an edge-coloured graph on $n$ vertices with $\delta^c(G) \ge \delta n+1$.
Then one of following statements holds
\begin{enumerate}[label={\rm(\alph*)}]
	\item $G$ contains a properly coloured $1$-path-cycle such that $|H| \ge \min \{n, ( 3 \delta + \alpha/2)n/2\}$ and every cycle in $H$ has length at least $\alpha n /100$;
	\item there exist disjoint $W^*, Z^* \subseteq V(G)$ such that 
\begin{enumerate}[label={\rm(\roman*)}]
	\item $ |W^*| \ge (\delta  - 7 \sqrt{\alpha} ) n $ and $|Z^*| \ge ( 2 \delta - 1 - 3 \alpha^{1/4} ) n$;
	\item for each $ w \in W^*$, there exists a distinct colour $c_w^*$ such that there are at least $|Z^*| - 3\sqrt{\alpha} n $ vertices $z \in Z^*$ such that $c(zw) = c_w^*$;
	\item for each $z \in Z^*$, $d_G(z) \le ( \delta + 4 \alpha^{1/4} ) n $ and there are at least $( \delta  - 6 \alpha^{1/4} ) n$ edges $zw$ such that $w \in W^*$ and $c(zw) = c_w^*$.
\end{enumerate}
\end{enumerate}
\end{lemma}

Here we give a brief description of the proof. 
By Lemma~\ref{lma:X1X2}, we may assume that $| X_1(H) \cup X_2(H) |$ is bounded below (or else (b) holds). 
Similarly $| Y_1(H) \cup Y_2(H) |$ is also bounded below.
Using Lemma~\ref{lma:H}, we then show that $|H| \ge (3 \delta + \alpha /2 ) n /2$ as desired.

\begin{proof}[Proof of Lemma~\ref{lma:AB}]
Let $\rho:= \alpha/1000$.
Let $H$ be a properly coloured $1$-path-cycle in $G$ such that every cycle in $H$ has length at least $\rho n $.
Suppose that $|H|$ is maximum.
We may assume that $|H| <  \min\{n, (3\delta + \alpha/2) n/2 \}$ or else we are done. 
By Lemma~\ref{lma:parameter}, we further assume that $H$ is a $1$-path-cycle with parameters $\rho$-$({\bf x}; {\bf y})$.

Let $X : = X_1(H) \cup X_2(H)$ and let $Y :=  Y_1(H) \cup Y_2(H)$.
By Lemma~\ref{lma:X1X2}, we may assume that $|X| \ge (\delta + \alpha) n$.
Similarly, by reversing all orientation of~$H$ and Lemma~\ref{lma:X1X2}, we may also assume that $|Y| \ge (\delta  + \alpha) n$.
Let $S_X$ be the set of vertices $v \in V(H)$ such that $(v,c_+(v)), (v,c_-(v)) \in X$.
Let $R_X : = \{ (x',c_{x'}) \in X : x' \notin S_X  \}$.
Note that 
\begin{align}
	2 |S_X| +|R_X|  = |X|  \ge   (\delta + \alpha) n .  \label{eqn:SRX}
\end{align}
Consider any ${\bf y'} = (y',c_{y'}) \in Y$.
Proposition~\ref{prop:X1} and Lemma~\ref{lma:H} imply that 
	\begin{align}
| N_G({\bf y}) | &\ge \delta n,  & 
N_G({\bf y}) & \subseteq V(H),&
|N_G ({\bf y}) \cap S_X| & \le 4 \rho n. \label{eqn:NGX}
\end{align}
If $R_X = \emptyset$, then 
\begin{align*}
|H| & \ge |N_G({\bf y'})| + |S_X| - 4 \rho n 
\ge \delta n + (\delta + \alpha) n/2  - 4 \rho n  
\ge (3\delta + \alpha/2) n/2,
\end{align*}
a contradiction.
Thus $R_X \ne \emptyset$.
Similarly, let $S_Y$ be the set of vertices $v \in V(H)$ such that $(v,c_+(v)) , (v,c_-(v)) \in Y$ and $R_Y : = \{ (y',c_{y'}) \in Y : y' \notin S_y  \}$.

Define $F$  to be the auxiliary directed bipartite graph on vertex classes $R_X$ and $R_Y$ such that there exists a directed edge from ${\bf v} = (v, c_v)$ to ${\bf w} =(w,c_w)$ if and only if 
\begin{itemize}
	\item $ \dist _G ( v , w ) \ge 2 \rho n $;
	\item $vw$ is an edge in $G$ with $c(vw) \ne c_v$.
\end{itemize}
By Lemma~\ref{lma:H}, $F$ is an oriented graph, that is, $F$ has no directed $2$-cycle. 
Consider any ${\bf y'} = (y',c_{y'}) \in Y$.
We have
\begin{align}
	d^+_F( {\bf y' } ) & \ge  |N_G({\bf y'}) \cap R_X| - 4 \rho n 
	 \ge  |N_G({\bf y'}) \cap ( R_X \cup S_X ) | - 4 \rho n - |N_G ({\bf y'}) \cap S_X| \nonumber\\
	& \overset{\mathclap{\eqref{eqn:NGX}}}{\ge} \delta n  + | R_X| + | S_X | - |H| - 8 \rho n  
	 \overset{\mathclap{\eqref{eqn:SRX}}}{\ge}  \frac{(3\delta  + \alpha - 16 \rho ) n  + |R_X|}{2} - |H| . \nonumber
\end{align}
Similarly, for any ${\bf x'} \in R_X$, $d^+_F( {\bf x' } ) \ge \frac{(3\delta  + \alpha - 16 \rho ) n  + |R_Y|}{2} - |H|$.
Since $F$ is an oriented graph, we have
\begin{align}
	|R_X| |R_Y| & \ge e(F) \ge  \sum_{{\bf x} \in R_X} d^+_F( {\bf x} ) + \sum_{{\bf y} \in R_Y} d^+_F( {\bf y} ) \nonumber \\
		& \ge |R_X| \left( \frac{(3\delta  + \alpha - 16 \rho ) n  + |R_Y|}{2} - |H| \right) + |R_Y| \left( \frac{(3\delta  + \alpha - 16 \rho ) n  + |R_X|}{2} - |H| \right) , \nonumber \\
		0 & \ge (|R_X| + |R_Y|) ( (3\delta  + \alpha - 16 \rho ) n/2 - |H| ) . \nonumber
\end{align}
This implies that $|H| \ge (3\delta  + \alpha - 16 \rho ) n/2 \ge (3\delta  + \alpha/2 ) n/2$ as $R_X \cup R_Y \ne \emptyset$, a contradiction. 
\end{proof}

When $\delta \ge 2/3$, Lemma~\ref{lma:AB} implies Lemma~\ref{lma:1-path-cycle}.
For $1/2 < \delta <2/3$, we present a rough sketch proof of Lemma~\ref{lma:1-path-cycle} using Lemma~\ref{lma:AB}.
Suppose that Lemma~\ref{lma:1-path-cycle} holds for any $\delta'$ with $\delta' > \delta$.
Apply Lemma~\ref{lma:AB} and we may assume that Lemma~\ref{lma:AB}(b) holds (or else we are done).
Thus there exist disjoint $Z^*, W^* \subseteq V(G)$ satisfying Lemma~\ref{lma:AB}(b).
Let $\delta^* : = ( \delta - 4 \alpha^{1/8} ) n / |G \setminus Z^*|$.
So $\delta^* > \delta$.
If $d^c(v, Z^*) \le  4 \alpha^{1/8} n$ for all vertices $v \notin Z^*$, then $\delta^c ( G \setminus Z^*) \ge ( \delta - 4 \alpha^{1/8} ) n  =  \delta^* |G \setminus Z^*|$.
Since $\delta^* > \delta$, we apply Lemma~\ref{lma:1-path-cycle} to $G \setminus Z^*$.
We have either a large enough properly coloured $1$-path-cycle or $G \setminus Z^*$ is $(\delta^*, \eps^*)$-extremal for some small $\eps^*$ or both.
In the second case, we then show that $G$ is $(\delta, \eps)$-extremal. 
This argument is formalised in the lemma below. 

We would need the following notation. 
For $\phi \ge 0 $, let $I_0(\phi) : = [ 2/3- \phi , 1)$.
For $s \in \mathbb{N}$, let $I_s(\phi) : = \{ p \in [0,1) \setminus \bigcup_{0 \le i < s} I_{i} (\phi) :  \frac{ p - \phi }{ 3/2 - p  } \in  I_{s-1}(\phi) \}$.
Let $s_{\phi}(\delta)$ be the integer $s$ such that $\delta \in I_s(\phi)$.

\begin{lemma} \label{lma:1-path-cycle2}
Let $0 < 1/n \ll \alpha_{ s_{ \phi } (\delta) } \ll \alpha_{ s_{ \phi } (\delta) -1} \ll \dots \ll   \alpha_0 \ll  \phi \ll\eps \ll 1/2 \ll \delta \le \delta^* < 1$.
Suppose that $4^{s_{\phi}(\delta)} \eps \ll \delta - 1/2$, and that
$G$ is a critical edge-coloured graph on~$n^* \ge 2^{s_{ \phi } (\delta^*)} n$ vertices with $\delta^c (G)  \ge \delta^* n ^* + 1$.
Then one of the following statements holds:
\begin{enumerate}[label={\rm(\roman*$^*$)}]
\item $G$ contains a properly coloured $1$-path-cycle $H$ such that $|H| \ge ( 3 \delta^* + \alpha_{ s_{ \phi } (\delta^*) } /2 ) n^* /2 $ and every cycle in $H$ has length at least $\alpha_{s_{\phi}(\delta^*)} n^* /100$;
\item $G$ is $(\delta^*, 4^{s_{\phi}(\delta^*)} \eps )$-extremal.
\end{enumerate}
%
%
%
%
\end{lemma}

\begin{proof}
Fix $\delta^*$ and write $s^*$ and $\alpha$ for $s_{\phi}(\delta^*)$ and $\alpha_{s_{\phi}(\delta^*)}$, respectively.
Without loss of generality, $\delta^* \le 2/3$.
Suppose that $G$ satisfies the hypothesis. 
Apply Lemma~\ref{lma:AB} to~$G$ with $\rho = \alpha_{s^*}/100$.
We may assume that Lemma~\ref{lma:AB}(b) holds or else we are done.
Thus there exist disjoint $W^*, Z^* \subseteq V(G)$ such that 
\begin{enumerate}[label={\rm(\roman*$'$)}]
	\item $ |W^*| \ge (\delta ^* - 7 \sqrt{ \alpha } ) n^* $ and $|Z^*| \ge ( 2 \delta ^* - 1 - 3  \alpha ^{1/4} ) n^*$;
	\item for each $ w \in W^*$, there exists a distinct colour $c_w^*$ such that there are at least $|Z^*| - 3\sqrt{\alpha} n^* $ vertices $z \in Z^* \cap N_G(w)$ such that $c(zw) = c_w^*$;
	\item for each $z \in Z^*$, $d_G(z) \le ( \delta^* + 4  \alpha^{1/4} ) n^* $ and there are at least $( \delta^*  - 6  \alpha  ^{1/4} ) n^*$ edges $zw$ such that $w \in W^* \cap N_G(z)$ and $c(zw) = c_w^*$.
\end{enumerate}

First suppose that $s^*=0$.
Since $\delta^* \ge 2/3 - \phi$ and $\alpha, \phi \ll \eps$, (i$'$) implies that
\begin{align*}
 |Z^*| \ge  ( 2 \delta^* - 1  - 3 \alpha^{1/4} ) n^*
=  ( 1 - \delta^* + (3 \delta^* - 2)  - 3 \alpha^{1/4} ) n^*
\ge ( 1 - \delta^* - \eps ) n^*.
\end{align*}
Thus $G$ is $(\delta^*, \eps )$-extremal.
So we may assume that $s \ge 1$ and the lemma holds for all $s' < s$.

Let $F$ be the subgraph of $G$ induced by edges $zv$ such that $z \in Z^*$ and either $v \notin W^*$ or $v \in W^*$ with $c ( z v ) \ne c_v$.
Note that by~(iii$'$), $e(F) \le 10 \alpha^{1/4}   n^* |Z^*| $.
Let $V_F$ be the set of vertices~$v$ such that $d_F(v) \ge 5 \alpha^{1/8}   n^* $.
So $|V_F| \le 5 \alpha^{1/8}n^*$.
For any $w \in W^*$, (i$'$) and (ii$'$) imply that
\begin{align}
d_{G}(w) & \ge (d^c_G(w) - 1) + |Z^*| - 3 \sqrt{\alpha} n^* 
 \ge (3 \delta^* -1 - 4 \alpha^{1/4}) n^* . \label{eqn:degW^*}
\end{align}

We split the proof into two cases depending on the value of~$\delta^*$.

\medskip
\noindent \textbf{Case 1: $ \delta^* < \frac{3 ( 1- 15 \alpha^{1/8})}{5(1- 10 \alpha^{1/8})}$.} 
Let $Z_1$ be a subset of $Z^*$ of size $|Z_1|	 =  ( \delta^* - 1/2 ) n^*  - |V_F|$
and let $Z_2 := Z^* \setminus Z_1$. 
Note that by (i$'$),
\begin{align}
|Z_2| \ge (  \delta^* - 1/2   - 3 \alpha^{1/4} ) n^*. \label{eqn:Z2}
\end{align}
Let $G' := G \setminus (Z_1 \cup V_F)$.
We claim that
\begin{align}
	\delta^c ( G' ) \ge  ( \delta^* - 10 \alpha^{1/8} ) n^*  +1\label{eqn:d^c(G')}
\end{align}
If $v \in V \setminus W^*$, then  $d^c_{G}(v,Z_1 \cup V_F ) \le d_{G}(v,Z^*) + |V_F| \le d_F(v) + |V_F| \le  10 \alpha^{1/8} n^*$.
If $w \in W^*$, then by (ii$'$), $d^c_{G}(v,Z_1 \cup V_F ) \le d^c_{G}(w,Z^*) + |V_F| \le 1 + 3 \sqrt{\alpha} n^* + |V_F| \le 10 \alpha^{1/8} n^*$.
Hence \eqref{eqn:d^c(G')} holds.

Let 
\begin{align}
n'  := |G'| = (3 / 2- \delta^*) n^*  \quad and \quad
 \delta'  := \frac{ \delta ^* - 10 \alpha^{1/8} }{3/2 - \delta ^*} \ge  \frac{ \delta ^* - \phi }{3/2 - \delta ^*} .
 \nonumber
\end{align}
Note that  $s_{\phi} (\delta') < s^*$, $\alpha n^* \ll \alpha_{s_{\phi} (\delta')} n'$ and  $\delta^c ( G' ) \ge \delta' n'+1$.
Also,
\begin{align*}
 \frac{(3\delta'+ \alpha_{s_{\phi} (\delta')} /2) n'}2 =  \frac{ 3 (\delta ^* - 10 \alpha^{1/8} )n^* + \alpha_{s_{\phi} (\delta')}  n'/2 }  2    
>  \frac{3 ( \delta^* + \alpha / 2 )  n^*}{2}.
\end{align*}
By our assumption on~$\delta^*$, we have $(3\delta'+ \alpha'/2)n'/ 2 < n'$.  
Clearly, $|G'| \ge n^* /2 \ge 2^{s_{\phi} (\delta')} n$.
Let $\eps' : = 4^{s_{\phi} (\delta')} \eps $.
By induction hypothesis, we may assume that $G'$ is $(\delta', \eps' )$-extremal (or else we are done).
Thus there exist disjoint $A',B' \subseteq V(G')$ such that
\begin{enumerate}[label={\rm(A\arabic*$'$)}]
	\item $|A'| \ge (\delta'- \eps')  n' $ and $|B'| \ge (1-\delta' - \eps')n'$;
	\item for each $ a \in A'$, there exists a distinct colour $c_a'$ such that there are at least $|B'| -  \eps' n'$ vertices $b \in B'$ such that $c(ab) = c_a'$;
	\item for each $b \in B'$, $d_G(b) \le (\delta' +  \eps' ) n' $ and $b$ has at least $|A'|-  \eps' n'$ neighbours $a \in A'$ such that $c(ab) = c'_a$. 
\end{enumerate}
Let $U' := V(G') \setminus (A' \cup B')$, so $|U'| \le 2 \eps' n'$.
Recall that $W^* \subseteq V(G')$ and that $\eps' , \alpha \ll \delta^* -1/2$.
For any $w \in W^*$,  
\begin{align}
d_{G'}(w) & \ge d_G(w) - |Z_1 \cup  V_F| 
\overset{\eqref{eqn:degW^*}}{\ge }(3 \delta^* -1 - 4 \alpha^{1/4}) n^* - ( \delta^* - 1/2)n^* \nonumber \\
& = (2 \delta^* - 1/2 - 4 \alpha^{1/4})n^*  \ge (\delta^* +\eps' )n^* \ge (\delta' + \eps')n'. \nonumber
\end{align}
Therefore $W^* \cap B' = \emptyset$ by~(A3$'$).
Let $A := W^* \cap A'$.
So 
\begin{align}
	|A| \ge |W^*| - |U'| 
	\overset{\text{(i$'$)}}{\ge} ( \delta^* - 7 \sqrt{\alpha} ) n^* - 2 \eps' n' 
	\ge (\delta^* - 4^{s^*} \eps )n^*  \label{eqn:|A|}
\end{align}
and $|A' \setminus A| \le (\delta' +\eps') n'- |A| \le 2 \cdot 4^{s^*} \eps n^*$.
Since $Z_2 \cap W^*= \emptyset$, we have $Z_2 \cap A' \subseteq A \cap A' $.
Hence
\begin{align*}
	| Z_2 \cap B'| \ge |Z_2| - | Z_2 \cap A'| - |Z_2 \setminus ( A' \cup B') | 
	\ge |Z_2| - |A \cap A'| - |U'| 
	\overset{\eqref{eqn:Z2}}{>} 3\sqrt{\alpha} n^* + \eps' n'.
\end{align*}
Consider any $a \in A$.
By~(ii$'$) and~(A2$'$),  there exists vertex $z \in Z_2 \cap B'$ such that $c_a^* = c(az) = c'_a $.
Therefore we have $c_a^* = c'_a$ for all $a \in A$. 

Let $B : = B' \cup Z_1$.
Note that 
\begin{align}
	|B| &  = |V(G) \setminus (A' \cup U' \cup V(F)|
	\ge n^* - |A'| - |U'| - |V_F| 
		\ge  ( 1- \delta - 4^{s^*} \eps ) n. \label{eqn:|B|}
\end{align}
We now claim that $G$ is $(\delta, 4^{s^*} \eps)$-extremal with partition~$(A,B)$.
Note that \ref{itm:A1} holds by \eqref{eqn:|A|} and~\eqref{eqn:|B|}.
Statements (ii$'$) and (A2$'$) imply~\ref{itm:A2}.
Similarly, statements (iii$'$) and (A3$'$) imply~\ref{itm:A3}.

\medskip
\noindent \textbf{Case 2: $ \delta^* \ge \frac{3 ( 1- 15 \alpha^{1/8})}{5(1- 10 \alpha^{1/8})}$.} 
Note that $s^*=1$.
Case~2 is proved via a similar argument used in Case~1, where we let $Z_1$ be a subset of $Z^*$ of size $|Z_1|	 =  (1 - (3 \delta^* + \alpha/2) /2 )n^*  - |V_F|$.
\end{proof}

We now prove Lemma~\ref{lma:1-path-cycle} by choosing $\phi, \alpha_0, \alpha_1, \dots, \alpha_{ s_{ \phi } (\delta) }$ appropriately. 

\begin{proof}[Proof of Lemma~\ref{lma:1-path-cycle}]
Let $s_0 :=  s_0(\delta)$ and let $\eps' := 4^{-2s_0}\eps$.
Choose $\beta \ll \phi \ll \eps', \delta - 1/2$ such that $s_{\phi}(\delta) \le 2s_0$.
So $ 4^{s_{\phi}(\delta)} \eps'\le \eps$.
Next choose  $\beta < \alpha_{ s_{ \phi } (\delta) } \ll \alpha_{ s_{ \phi } (\delta) -1} \ll \dots \ll   \alpha_0 \ll \phi$.
Therefore, Lemma~\ref{lma:1-path-cycle2} with $\eps'$ playing the role of $\eps$ implies Lemma~\ref{lma:1-path-cycle}.
\end{proof}

\section*{Acknowledgments}

The author would like to thank the referee for their valuable comments.

\end{document}